\documentclass[a4paper,reqno]{amsart}

\usepackage{amssymb}
\usepackage{amsfonts}
\usepackage{amsmath}
\usepackage{graphicx}
\usepackage{color}
\usepackage{hyperref}
\usepackage{tikz}
\usetikzlibrary{matrix}
\usepackage{CMMO-private}
\usepackage{rotating}
\usepackage{mathtools}
\mathtoolsset{showonlyrefs}

\providecommand{\U}[1]{\protect\rule{.1in}{.1in}}
\numberwithin{equation}{section}
\numberwithin{equation}{section}
\numberwithin{equation}{section}
\newtheorem{definition}{Definition}[section]
\newtheorem{theorem}[definition]{Theorem}
\newtheorem{lemma}[definition]{Lemma}
\newtheorem{proposition}[definition]{Proposition}
\newtheorem*{theorem*}{Theorem}

\theoremstyle{definition} {\newtheorem{remark}[definition]{Remark}}

\makeindex
\begin{document}

\title[Dimension reduction and structured deformations]{Dimension reduction in the context of structured deformations}
\author{Gra\c{c}a Carita}
\author{Jos\'{e} Matias}
\address{Departamento de Matem\'atica, Instituto Superior T\'ecnico, Av.\@ Rovisco Pais, 1, 1049-001 Lisboa, Portugal}
\email[J.~Matias]{jose.c.matias@tecnico.ulisboa.pt}
\author{Marco Morandotti}
\address{Technische Universit\"at M\"unchen, Boltzmannstrasse 3, 85748 Garching b.\@ M\"unchen, Germany}
\email[M.~Morandotti \myenv]{marco.morandotti@ma.tum.de}
\author{David R.\@ Owen}
\address{Department of Mathematical Sciences, Carnegie Mellon University, 5000 Forbes Ave., Pittsburgh, 15213 USA}
\email[D.~R.~Owen]{do04@andrew.cmu.edu}
\date{\today}

\begin{abstract}
In this paper we apply both the procedure of dimension reduction and the incorporation of structured deformations to a three-dimensional continuum in the form of a thinning domain.
We apply the two processes one after the other, exchanging the order, and so obtain 
for each order both a relaxed bulk and a relaxed interfacial energy.
Our implementation requires some substantial modifications of the two relaxation procedures.
For the specific choice of an initial energy including only the surface term, we compute the energy densities explicitly and show that they are the same, independent of the order of the relaxation processes.
Moreover, we compare our explicit results with those obtained when the limiting process of dimension reduction and of passage to the structured deformation is carried out at the same time.
We finally show that, in a portion of the common domain of the relaxed energy densities, the simultaneous procedure gives an energy strictly lower than that obtained in the two-step relaxations.
\end{abstract}

\keywords{Dimension reduction, structured deformations, relaxation, integral representation of functionals, explicit formulas}
\subjclass[2010]{
49J45, 
(74Kxx, 
74A60, 
74G65
)}
\dedicatory{Dedicated to our friend and colleague Gra\c{c}a Carita, who left us far too soon.}

\maketitle

\tableofcontents

\section{Introduction}\label{intro}
Classical continuum theories of elastic bodies are amenable to refinements that broaden their range of applicability or that adapt them to specific physical contexts. 
In this article we consider refinements that (i) incorporate into a classical theory the effects of submacroscopic slips and separations (disarrangements) or that (ii) adapt the theory to the description of thin bodies. 
Refinements of the type (i) are intended to describe finely layered bodies such as a stack of papers, granular bodies such as a pile of sand, or bodies with defects such as a metal bar. 
Those of type (ii) are intended to provide descriptions of membranes such as a sheet of rubber, descriptions of thin plates such as a sheet of metal, and descriptions of fibered thin bodies such as a sheet of paper. 
There are available a variety of approaches for incorporating disarrangements and for adaptation to the case of thin bodies: examples of refinements of type (i) are mechanical theories of no-tension materials \cite{angelillo,DO,LSZ}, of granular media \cite{numberone,K,Mue}, of single and polycrystals \cite{KS4,R2}, and of elastic bodies in the multiscale geometrical setting of structured deformations \cite{DO1,owen}, while for refinements of type (ii) the method of dimension reduction via $\Gamma$-convergence \cite{BF2001,LDR95,LDR96} and the method of dimension reduction via Taylor expansions \cite{DPZ} provide examples.

Our goal in this paper is to implement in succession refinements of both types, starting from a classical, energetic description of three-dimensional elastic bodies. 
Specifically, for a refinement of type (i) we choose the context of structured deformations to incorporate the effects of submacroscopic slips and separations into a refined energetic response, while for a refinement of type (ii) we employ the method of dimension reduction via $\Gamma$-convergence to obtain a refined energetic response. 
With the starting point a three-dimensional body with a given energetic response, the two types of refinements can be carried out in two different orders, and each order of applying the two types of refinements will result in an energetic description of a two-dimensional body undergoing
submacroscopic disarrangments, as indicated below in Figure \ref{figzero}:

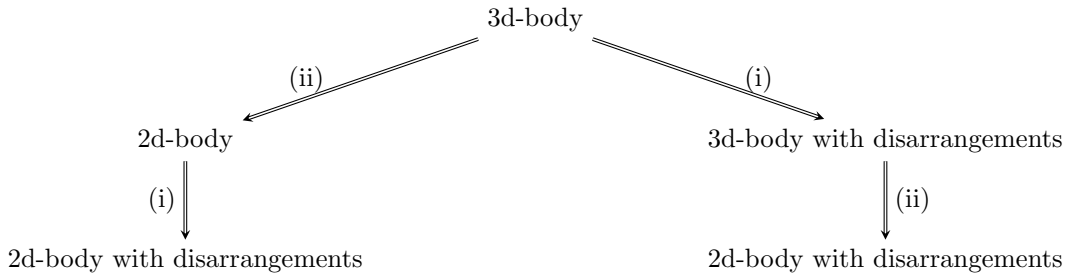
\begin{figure}[h]
\begin{center}
\begin{tikzpicture}
  \matrix (m) [matrix of math nodes,row sep=3em,column sep=4em,minimum width=2em]
  {
    & \text{3d-body} & \\
    \text{2d-body} & & \text{3d-body with disarrangements}  \\
    \text{2d-body with disarrangements} & & \text{2d-body with disarrangements} \\
    };
  \path[-stealth]
    (m-1-2) edge [double] node [left] {$\text{(ii)}\quad$} (m-2-1)
            edge [double] node [right] {$\quad\text{(i)}$} (m-2-3)
    (m-2-1) edge [double] node [left] {$\text{(i)}$} (m-3-1)
    (m-2-3) edge [double] node [right] {$\text{(ii)}$} (m-3-3) ;
\end{tikzpicture}
\caption{The two paths for refinements of classical continuum theories: (i) structured deformations (SD) and (ii) dimension reduction (DR).}
\label{figzero}
\end{center}
\end{figure}

The right-hand path above begins with the incorporation of disarrangements (i) and then applies dimension reduction (ii), while the left-hand path reverses the order. 
We consider in this paper the nature of the energetic responses obtained at each step in the two paths and whether or not the two-dimensional body with disarrangements obtained via the left-hand path above has the same energetic response as that obtained via the right-hand
path. 

Incorporation of disarrangements via structured deformations (i) replaces a vector field $u$ that maps a three-dimensional body into three-dimensional space by a pair $(g,G)$, where $g$ also maps the three-dimensional body into three-dimensional space and $G$ is a matrix-valued field that gives the contributions at the macroscopic level of submacroscopic deformations without disarrangements. 
The matrix-valued field $\nabla g-G$ then gives the contributions at the macroscopic level of submacroscopic deformations due to disarrangments. 
Dimension reduction (ii) replaces the vector field $u$ by a pair $(\cl{u},\cl{d})$ of vector fields defined on a two-dimensional body, where $\cl{u}$ places the two-dimensional body into three-dimensional space and $\cl{d}$ is a "director field" on the two-dimensional body that is a geometrical residue of the passage from a three-dimensional body to a two-dimensional body.
In the diagram above, both (i) and (ii) begin with one and the same energy that depends only upon the field $u$: (i) results in an energy that depends upon the pair $(g,G)$, while (ii) results in an energy that depends on the pair $(\cl{u},\cl{d})$. 
When (i) and (ii) are applied consecutively, in either order, the resulting energy depends on a triple of fields $(\cl{g},\cl{G},\cl{d})$ defined on a two-dimensional body. 
The mathematical properties of these fields and the relation between the energy responses at each stage are summarized in the remainder of this introduction.

\subsection{Statement of the problem and results}\label{sect:statement}
Let $\omega\subset\mathbb{R}^{2}$ be a bounded open set, let $\eps>0,$ and let $\Omega_{\eps}:=\omega\times(-\frac{\eps}{2},\frac{\eps}{2})$. 
We recall that the set of \emph{special functions of bounded variation} on $\Omega_\eps$ consists of those $BV$ functions whose distributional derivative has no Cantor part, namely $SBV(\Omega_\eps;\R3):=\{u\in BV(\Omega_\eps;\R{3}): D^cu=0\}$ (see Section \ref{sect:BV}).
For a function $u\in SBV(\Omega_{\eps};\mathbb{R}^{3})$, consider the energy%
\begin{equation}
E_{\eps}(u):=\int_{\Omega_{\eps}}W_{3d}(\nabla u(x))\,\de x+\int_{\Omega_{\eps}\cap S(u)}h_{3d}\big([u](x),\nu(u)(x)\big)\,\de\mathcal{H}^{2}(x)\label{E3d}%
\end{equation}
where $S(u)$ is the jump set of $u, [u]$ is the jump of $u$ across $S(u)$, and $\nu(u)$ is the unit normal vector to $S(u)$. 
The volume and surface energy densities $W_{3d}\colon{\mathbb{R}}^{3\times3}\rightarrow\lbrack0,+\infty)$ and $h_{3d}\colon{\mathbb{R}}^{3}\times \S{2}\rightarrow\lbrack0,+\infty)$ are continuous functions satisfying the following hypotheses:
\begin{itemize}
	\item[$(H_{1})$] There exists a constant $c_W>0$ such that growth conditions from above and below are satisfied%
	\begin{align}
	\frac{1}{c_W} |A|^p & \leq  {W}_{3d}(A), \label{coerc} \\
	|W_{3d}(A)-W_{3d}(B)| & \leq  c_W|A-B|(1+|A|^{p-1}+|B|^{p-1}), \label{growth} 
	\end{align}
	for any $A,B \in\mathbb{R}^{3\times3},$ and for some $p >1$.
\item[$(H_{2})$] There exists a constant $c_h>0,$ such that for all
	$(\lambda,\nu)\in\mathbb{R}^{3}\times \S{2}$%
	$$\frac1{c_h}|\lambda|\leq h_{3d}(\lambda,\nu)\leq c_{h}|\lambda|.$$
\item[$(H_{3})$] $h_{3d}(\cdot,\nu)$ is \emph{positively $1$-homogeneous}: for all
	$t>0$, $\lambda\in\R{3}$
	$$h_{3d}(t\lambda,\nu)=t\,h_{3d}(\lambda,\nu).$$
\item[$(H_{4})$] $h_{3d}(\cdot,\nu)$ is \emph{subadditive}: for all
	$\lambda_{1},\lambda_{2}\in\mathbb{R}^{3}$
	$$h_{3d}(\lambda_{1}+\lambda_{2},\nu)\leq h_{3d}(\lambda_{1},\nu)+h_{3d}(\lambda_{2},\nu).$$
\end{itemize}

\begin{remark}\label{energies} 
\textit{(i)} The coercivity condition \eqref{coerc} in ($H_1$), although useful to obtain $L^p$ boundedness of the gradients, is not physically desirable.
It can be removed following the argument in \cite[proof of Proposition 2.22, Step 2]{CF}: if $W_{3d}$ is not coercive, one can consider $W_{3d}^\beta(\cdot):=W_{3d}(\cdot)+\beta|\cdot|^p$ and then take the limit as $\beta\to0$. \\
\textit{(ii)} By fixing $B$ in \eqref{growth}, one can easily show that $W_{3d}$ satisfies also a growth condition of order $p$, that is, there exists a constant $C>0$ such that for every $A\in\R{3\times3}$
\begin{equation}\label{pgrowth}
W_{3d}(A)\leq C(1+|A|^p).
\end{equation}
\end{remark}

Under assumptions $(H_1)$--$(H_4)$, we carry out both a procedure of dimension reduction as $\eps\to0$ to obtain an energy functional defined on the cross--section $\omega$, and a procedure of relaxation to obtain an energy functional defined on structured deformations. 
The two procedures performed consecutively result in a doubly relaxed energy that may depend upon the order chosen.
We will perform the two processes in both possible orders and compare the doubly relaxed energies.
The schematic description of the two possible orders in Figure \ref{figzero} now takes the form in Figure \ref{figone}.
\begin{figure}[h]
\begin{center}
\begin{tikzpicture}
  \matrix (m) [matrix of math nodes,row sep=3em,column sep=4em,minimum width=2em]
  {
    & W_{3d}, h_{3d} & \\
    W_{3d,2d}, h_{3d,2d} & & W_{3d,SD}, h_{3d,SD} \\
    W_{3d,2d,SD}, h_{3d,2d,SD} & & W_{3d,SD,2d}, h_{3d,SD,2d} \\
    };
  \path[-stealth]
    (m-1-2) edge [double] node [left] {$\mathrm{DR}\quad$} (m-2-1)
            edge [double] node [right] {$\quad\mathrm{SD}$} (m-2-3)
    (m-2-1) edge [double] node [left] {$\mathrm{SD}$} (m-3-1)
    (m-2-3) edge [double] node [right] {$\mathrm{DR}$} (m-3-3) ;
\end{tikzpicture}
\caption{Energy densities for the two paths for dimension reduction (DR) and structured deformations (SD).}
\label{figone}
\end{center}
\end{figure}
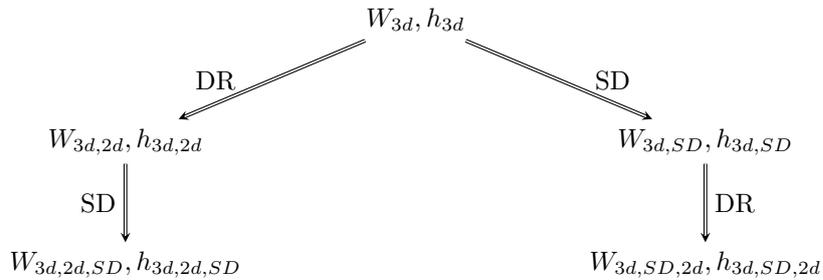
As indicated in Figure \ref{figone}, we derive formulas for bulk and interfacial densities obtained on the left-hand path and for those obtained via the right-hand path. 

The technical background for structured deformations and dimension reduction can be found in the following literature:
\begin{itemize}
\item[(i)] for structured deformations, we use the techniques introduced in \cite{CF}, where the relaxation process is obtained by combining the blow-up method of \cite{FM} with the construction of suitable approximating sequences by means of Alberti's theorem \cite{AL};
\item[(ii)] for dimension reduction, we employ the classical approach \cite{LDR95,LDR96} of rescaling the spatial variable to write the energy in the domain $\Omega=\omega\times(-1/2,1/2)$ and rescale the energy by dividing it by $\eps$.
\end{itemize}
Nevertheless, the sequential application of (i) and (ii) one after the other in both orders requires some non-trivial adaptations which are detailed in Remark \ref{independent}.

We now summarize in abbreviated form the main results of this paper, and we refer the reader to Sections \ref{sect:LHS} and \ref{sect:RHS} for more detailed versions.

\paragraph{\textbf{The left-hand path}}
According to the diagram in Figure \ref{figone}, we first perform (DR) and then (SD).
Following (ii), we relax the energies $F_\eps(u):=\frac1\eps E_\eps(u)$ defined for $u\in SBV(\Omega_\eps;\R3)$ to an energy $\cF_{3d,2d}(\cl u,\cl d)$ defined for pairs $(\cl u,\cl d)\in SBV(\omega;\R3)\times L^p(\omega;\R3)$.
The deformation $\cl u(x_\alpha)$ is the limit of deformations $u_n(x_\alpha,x_3)$, where the dependence on the out-of-plane variable vanishes in the limit, whereas the vector $\cl d$ emerges as a weak limit of the out-of-plane deformation gradient \cite{BFM,BFM1}. 
\begin{theorem*}[Theorem \ref{T3d2d}]
Given a pair $(\cl u,\cl d)\in SBV(\omega;\R3){\times}L^p(\omega;\R3)$, the relaxed energy $\cF_{3d,2d}(\cl u,\cl d)$ defined in \eqref{107a} admits the integral representation \eqref{F3d2d}, where the relaxed energy densities $W_{3d,2d}\colon\R{3\times2}\times\R3\to[0,+\infty)$ and $h_{3d,2d}\colon\R3{\times}\S1\to[0,+\infty)$ are given by \eqref{107b} and \eqref{107c}, respectively.
\end{theorem*}
In Proposition \ref{p100} we prove that the densities $W_{3d,2d}$ and $h_{3d,2d}$ satisfy the hypotheses of the relaxation method for structured deformations of \cite{CF}, which leads to the relaxed energy $\cF_{3d,2d,SD}(\cl g,\cl G,\cl d)$ defined for $(\cl g,\cl G,\cl d)\in SBV(\omega;\R3){\times} L^1(\omega;\R{3{\times}2}){\times} L^p(\omega; \R3)\to[0,+\infty)$ and to the following theorem.
\begin{theorem*}[Theorem \ref{secondleft}]
Given a triple $(\cl g,\cl G,\cl d)\in SBV(\omega;\R3)\times L^1(\omega;\R{3{\times}2})\times L^p(\omega; \R3)$, the relaxed energy $\cF_{3d,2d,SD}(\cl g,\cl G,\cl d)$ defined in \eqref{200} admits the integral representation \eqref{201}, where the relaxed energy densities $W_{3d,2d,SD}\colon\R{3\times2}{\times}\R{3\times2}{\times}\R3\to[0,+\infty)$ and $h_{3d,2d,SD}\colon\R3{\times}\S1\to[0,+\infty)$ are given by \eqref{2011} and \eqref{2012}, respectively.
\end{theorem*}

\paragraph{\textbf{The right-hand path}}
According to the diagram in Figure \ref{figone}, we first perform (SD) and then (DR).
Following (i), the assumptions $(H_1)$--$(H_4)$ allow us to apply directly \cite[Theorem 2.17]{CF} to obtain a representation theorem for the relaxed energy $\cF_{3d,SD}\colon SBV(\Omega;\R3)\times L^1(\Omega;\R{3\times2})\to[0,+\infty)$ defined for structured deformations $(g,G^{\backslash3})$.
Strictly speaking, the structured deformations under consideration are pairs $(g,(G^{\backslash3}|\nabla_3 g))\in SBV(\Omega;\R3)\times L^1(\Omega;\R{3\times3})$, where for each $A\in\R{3\times2}$ and $q\in\R{3\times1}$, $(A|q)\in\R{3\times3}$ is formed from the two columns of $A$ and the single column of $q$.
Because the $3\times3$ matrix values of $(g,(G^{\backslash3}|\nabla_3 g))$ are determined by the pair of fields $(g,G^{\backslash3})$, we allow this abuse of terminology and notation.
\begin{theorem*}[Theorem \ref{firstSD}]
Given a pair $(g,G^{\backslash3})\in SBV(\Omega;\R3)\times L^1(\Omega;\R{3\times2})$, the relaxed energy density $\cF_{3d,SD}(g,G^{\backslash3})$ defined in \eqref{2018} admits the integral representation \eqref{2017}, where the relaxed energy densities $W_{3d,SD}\colon\R{3\times3}\times\R{3\times2}\to[0,+\infty)$ and $h_{3d,SD}\colon\R3\times\S2\to[0,+\infty)$ are given by \eqref{2015} and \eqref{2016}, respectively.
\end{theorem*}
Proposition \ref{energiesCF} collects some properties of the densities $W_{3d,SD}$ and $h_{3d,SD}$.
Therefore, performing the dimension reduction on the energy $\cF_{3d,SD}$ leads to the definition of the energy $\cF_{3d,SD,2d}\colon$ $SBV(\omega;\R3)\times L^1(\omega;\R{3{\times}2})\times L^p(\omega; \R3)\to[0,+\infty)$ for triples $(\cl g,\cl G,\cl d)$ defined on the cross--section $\omega$, and to the following theorem.
\begin{theorem*}[Theorem \ref{secondright}]
Given a triple $(\cl g,\cl G,\cl d)\in SBV(\omega;\R3)\times L^1(\omega;\R{3{\times}2})\times L^p(\omega; \R3)$, the relaxed energy $\cF_{3d,SD,2d}(\cl g,\cl G,\cl d)$ defined in \eqref{2019} admits the integral representation \eqref{301}, where the relaxed energy densities $W_{3d,SD,2d}\colon$ $\R{3\times2}{\times}\R{3\times2}{\times}\R3\to[0,+\infty)$ and $h_{3d,SD,2d}\colon\R3{\times}\S1\to[0,+\infty)$ are given by \eqref{2013} and \eqref{2014}, respectively.
\end{theorem*}
In this case, the results follow from those in \cite{CF} for the structured deformation part (Theorem \ref{firstSD}), and from applying Theorem \ref{T3d2d} for the dimension reduction part (Theorem \ref{secondright}).

\begin{remark}\label{independent}
We want to stress here that in both paths, the relaxation due to dimension reduction and that due to structured deformations are distinct refinements of the classical energetics of elastic bodies: the first one gives rise to the vector field $\cl d$ in the energy $\cF_{3d,2d}(\cl u,\cl d)$, whereas the second one gives rise to the matrix-valued field $G$ in the energy $\cF_{3d,SD}(g,G)$.
Nevertheless, the consecutive application of the two refinements requires some non-trivial adaptations of the existing relaxation techniques underlying dimension reduction and structured deformations.
Specifically, the use of $\cF_{3d,2d}(\cl u,\cl d)$ as an initial energy for relaxation in the context of structured deformations requires that $\cl u$ be a special function of bounded variation rather than a Sobolev function.
Therefore, the dimension reduction has to be carried out in the $SBV$ setting.
Moreover, the presence of $\cl d$ in the initial energy $\cF_{3d,2d}(\cl u,\cl d)$ for the (SD) relaxation requires a new modification of the standard relaxation techniques for structured deformations (see \cite{MMZ}).

In addition, in order to connect with standard applications of dimension reduction results, the inclusion of $\cl d$ puts an additional constraint on the approximating sequences $\{u_n\}$ in the (DR) relaxation, namely $\tfrac1{\eps_{n}}\int_{I}\nabla_{3}u_{n}\,\de x_{3}\rightharpoonup \overline{d} \text{ in }L^{p}(\omega;\R3)$, see \eqref{107a}.

The novelty of our approach lies partly in the incorporation of both the lack of smoothness of the function $\cl u$ (as in \cite{BF2001}) and the constraint $\tfrac1{\eps_{n}}\int_{I}\nabla_{3}u_{n}\,\de x_{3}\rightharpoonup \overline{d} \text{ in }L^{p}(\omega;\R3)$ in \eqref{107a} on the approximating sequence $\{u_n\}$ (as in \cite{BFM}), and partly in the modifications required to apply the standard (SD) relaxation introduced in \cite{CF} (see also \cite{BBBF}).
Moreover, the condition $\nu(u_n)\cdot e_3=0$ in \eqref{107a} for the left-hand path and $\nu(g_n)\cdot e_3=0$ in \eqref{2019} for the right-hand path rule out the occurrence of slips and separations on surfaces with normal parallel to the thinning direction $e_3$ and place an additional constraint on the process of dimension reduction.
The restriction in Theorem \ref{firstSD} to structured deformations of the form $(g,(G^{\backslash 3}|\nabla_3 g))$ is made in the same spirit for the right-hand path, since it implies that the disarrangement matrix $\nabla g-(G^{\backslash 3}|\nabla_3 g)$ has third column zero.
\end{remark}

The paper is organized as follows: in Section \ref{sect:prel} we set the notation and we recall some known results that are useful in the sequel, especially about $BV$ functions and $\Gamma$-convergence.
In Section \ref{sect:LHS}, we follow the left-hand path of Figure \ref{figone}; namely we first derive an energy on the cross--section $\omega$ and then we relax the energy to obtain one defined on structured deformations.
In Section \ref{sect:RHS}, we follow the right-hand path in Figure \ref{figone}: we first relax the energy to structured deformations and then we perform the dimension reduction.

In Section \ref{sect:comp_ex}, we compare the two doubly relaxed energies from the left-hand and right-hand paths for a specific initial energy which is purely interfacial. 
Setting $W_{3d}\equiv0$ and $h_{3d}(\lambda,\nu)=|\lambda\cdot\nu|$ in \eqref{E3d}, we show that the two paths lead to the same relaxed energy.
In particular, in Proposition \ref{S1} we give explicit formulas for the energies provided by Theorems \ref{secondleft} and \ref{secondright}, thus showing that they are equal.

In Section \ref{sect:6}, we present the alternative relaxation procedure of \cite{MS} in which the introduction of disarrangements and the thinning of the domain occur simultaneously.
For the specific choice of initial energy made in Section \ref{sect:comp_ex}, we prove that the relaxed energy from the scheme of \cite{MS} is identically zero.

In Section \ref{conclusions}, we summarize the main results of this research and provide an outlook for future research.

\section{Preliminaries}\label{sect:prel}
The purpose of this section is to give a brief overview of the concepts and results that are used in the sequel. 
Almost all these results are stated without proof as they can be readily found in the references given below.

\subsection{Notation}
Throughout the manuscript, the following notation will be employed:
\begin{itemize}
\item[-] $\omega \subset \R2$ is a open bounded set and for $ 0<\eps\leq1$, $\Omega_{\eps} := \omega \times(-\frac{\eps}{2}, \frac{\eps}{2})$; moreover, we denote $\Omega_1$ by $\Omega$ and notice that $\Omega =  \omega \times I$, with $I:= (-\frac{1}{2}, \frac{1}{2})$.
\item[-] given a vector $v \in \R3$, we write $v:=(v_{\alpha},v_3)$, where $ v_\alpha:= (v_1, v_2) \in \R2$ is the vector of the first two components of $v$;
\item[-] ${\mathcal A}(\Omega)$ (resp. ${\cA}(\omega)$) is the family of all open subsets of $\Omega $ (resp. $ \omega$);
\item[-] for all $A,B\in\cA(\Omega)$ (resp. ${\cA}(\omega)$), $A\Subset B$ means that there exists a compact subset $C$ of $\Omega$ (resp. $\omega$) such that $A\subset C\subset B$;
\item[-] $\cM (\Omega)$ (resp. $\cM(\omega)$) is the set of finite Radon measures on $\Omega$ (resp. $\omega$);
\item[-] $\cL^{N}$ and $\cH^{N-1}$ stand for the  $N$-dimensional Lebesgue measure  and the $\left(  N-1\right)$-dimensional Hausdorff measure in $\R N$, respectively;
\item[-] $\lVert\mu\rVert$ stands for the total variation of a measure  $\mu\in \cM (\Omega)$ (resp. $\cM (\omega)$);
\item[-] $\S{N-1}$ stands for the unit sphere in $\R N$;
\item[-]  $Q:=I^3$ and $Q':=I^2$ denote the unit cubes centered at the origin of $\R3$ and $\R2$, respectively; 
\item[-] $Q_\eta$ (resp. $Q'_\eta$) denotes the unit cube of $\R3$ (resp. $\R2$) centered at the origin with two sides perpendicular to the vector $\eta\in\S2$ (resp. $\eta\in\S1$);
\item[-] $Q(x, \delta):=x+\delta Q$, $Q_\eta(x, \delta):=x+\delta Q_\eta$ in $\R3$, and $Q'(x, \delta):=x+\delta Q'$, $Q'_\eta(x, \delta):=x+\delta Q'_\eta$ in $\R2$; 
\item[-] for $\eta \in \S{1}$, we define $\tilde{\eta} \in \S2$ by $\tilde{\eta} := (\eta, 0)$;
\item[-]  $C$ represents a generic positive constant that may change from line to line;
\item[-] $\lim_{\delta, n} := \lim_{\delta \to 0^+} \lim_{n \to \infty}, \; \lim_{k,n} := \lim_{k \to \infty} \lim_{n \to \infty}$.
\end{itemize}

\subsection{$BV$ functions}\label{sect:BV}
We start by recalling some facts on functions of bounded variation which will be used afterwards. We refer to \cite{AFP} and the references therein for a detailed theory on this subject.

Only in this subsection, $\Omega$ denotes a generic open set in $\R{N}$. 
A function $u \in L^1(\Omega; \R d)$ is said to be of {\em bounded variation}, and we write $u \in BV(\Omega; \R d)$, if its first distributional derivatives $D_j u_i$ are in $\cM (\Omega)$ for  $i=1,\ldots,d$ and  $j=1,\ldots,N.$ 
The matrix-valued measure whose entries are $D_j u_i$ is denoted by $Du.$ The space $BV(\Omega; \R d)$ is a Banach space when endowed with the norm
$$\lVert u\rVert_{BV} := \lVert u\rVert_{L^1} + \lVert Du\rVert(\Omega).$$
By the Lebesgue Decomposition theorem $Du$ can be split into the sum of two mutually singular measures $D^{a}u$ and $D^{s}u$ (the absolutely continuous part  and the singular part, respectively, of $Du$ with respect to the Lebesgue measure $\mathcal{ L}^N$). 
By $\nabla u$ we denote the Radon-Nikod\'{y}m derivative of $D^{a}u$ with respect to $\mathcal L^N$, so that we can write
$$Du= \nabla u \mathcal L^N \res \Omega + D^{s}u.$$

Let $\Omega_u$ be the set of points where the approximate limits of $u$ exists and $S(u)$ the {\it jump set} of this function, i.e., the set of points $x\in \Omega\setminus \Omega_u$ for which there exists $a, \,b\in \R N$ and  a unit vector $\nu  \in \S{N-1}$, normal to $S(u)$ at $x$, such that $a\neq b$ and
\begin{equation} \label{jump1} 
\lim_{\delta \to 0^+} \frac {1}{\delta^N} \int_{\{ y \in Q_{\nu}(x,\delta) : (y-x)\cdot\nu  > 0 \}} | u(y) - a| \,\de y=0
\end{equation}
and
\begin{equation}\label{jump2} 
\lim_{\delta \to 0^+} \frac {1}{\delta^N} \int_{\{ y \in Q_{\nu}(x,\delta) : (y-x)\cdot\nu  < 0 \}} | u(y) - b| \,\de y = 0.
\end{equation}
The triple $(a,b,\nu)$ is uniquely determined by \eqref{jump1} and \eqref{jump2}, up to permutation of $(a,b)$ and a change of sign of $\nu$, and it is denoted by $\left(u^+ (x),u^- (x),\nu(u) (x)\right)$.

If $u \in BV(\Omega;\R{d})$ it is well known that $S(u)$ is countably $(N-1)$-rectifiable, i.e.,
$$S(u) = \bigcup_{n=1}^{\infty}K_n \cup K_0,$$
where ${\mathcal H}^{N-1}(K_0) = 0$ and $K_n$ are compact subsets of $C^1$ hypersurfaces.
Furthermore,  ${\mathcal H}^{N-1}((\Omega\setminus \Omega_u) \setminus S(u)) = 0$ and the following decomposition holds
$$Du= \nabla u \mathcal L^N \res \Omega + [u] \otimes \nu(u) {\mathcal H}^{N-1}\res S(u) + D^c u,$$
where $[u]:= u^+ - u^-$ and $D^c u$ is the Cantor part of the measure $Du,$ i.e., $D^c u= D^{s}u\res (\Omega_u)$.

The space of \emph{special functions of bounded variation}, $SBV(\Omega; \R d)$, introduced in \cite{DGA} to study free discontinuity problems, is the space of functions $u \in BV(\Omega; \R d)$ such that $D^cu = 0$, i.e. for which
$$ Du = \nabla u \cL^N + [u] \otimes \nu(u)  \cH^{N-1} \res S(u).$$
We next recall some properties of BV functions used in the sequel. We start with the following lemma whose proof can be found in \cite{CF}.
\begin{lemma}\label{ctap}
Let $u \in BV(\Omega; \R d)$. There exists a sequence of piecewise constant functions ${u_n}\in SBV(\Omega;\R{d})$  such that $u_n \to u$ in $L^1(\Omega; \R d)$ and
 $$\lVert Du\rVert(\Omega) = \lim_{n\to \infty}\lVert Du_n\rVert(\Omega) = \lim_{n\to \infty} \int_{S (u_n)} |[u_n](x)|\; \de{\mathcal H}^{N-1}(x).$$
 \end{lemma}
The next result is a Lusin-type theorem for gradients due to  Alberti \cite{AL} and is essential to our arguments.
\begin{theorem}\label{Al}
Let $f \in L^1(\Omega; \R{d{\times} N})$. 
Then there exist $u \in SBV(\Omega; \R d)$ and a Borel function $g: \Omega\to\R{d{\times} N}$ such that
$$ Du = f {\cL}^N + g {\mathcal H}^{N-1}\res S(u),$$
$$ \int_{S(u)} |g| \; \de \cH^{N-1}(x) \leq C \lVert f\rVert_{L^1(\Omega; \R{d {\times} N})},$$
for some constant $C>0$.
Moreover, $\lVert u\rVert_{L^1(\Omega;\R d)} \leq 2 C\lVert f\rVert_{L^1(\Omega; \R{d {\times}N})}$.
\end{theorem}

\subsection{$\Gamma$-convergence and relaxation}\label{Gc}
We recall now the basics of $\Gamma$-convergence: this is a notion of convergence, introduced by De Giorgi and Franzoni \cite{DGF}, which is useful in the calculus of variations.
It allows one to study the convergence of (sequences of) variational functions by identifying their variational limit.
One of the most important products of the theory of $\Gamma$-convergence is the convergence of minima (see Remark \ref{PGC}).
We refer the reader to \cite{Braides,DM} for treatises on the topic and we collect here the most important definitions and results.

Let $X$ be a metric space and let $\{F_n\}$ be a sequence of functions $F_n\colon X\to\overline{\R{}}$ with values in the extended reals $\overline{\R{}}$.
\begin{definition}[{\cite[Definition 1.5]{Braides}}]\label{GammaCVG}
We say that the sequence $\{F_n\}$ \emph{$\Gamma$-converges} in $X$ to $F\colon X\to\overline{\R{}}$ if for all $x\in X$ we have
\begin{itemize}
\item[(i)] ($\liminf$ inequality) for every sequence $\{x_n\}$ converging to $x$
\begin{equation}\label{BGLI}
F(x)\leq\liminf_{n\to\infty} F_n(x_n);
\end{equation}
\item[(ii)] ($\limsup$ inequality) there exists a sequence $\{x_n\}$ converging to $x$ such that
\begin{equation}\label{BGLS}
F(x)\geq\limsup_{n\to\infty} F_n(x_n).
\end{equation}
\end{itemize}
The function $F$ is called the \emph{$\Gamma$-limit} of $\{F_n\}$, and we write $F=\Gamma-\lim_{n\to\infty} F_n$.
\end{definition}
When $X$ is an arbitrary topological space (in particular, it is not a metric space), a more general, \emph{topological}, definition of $\Gamma$-convergence can be given in terms of the topology of $X$.
We refer the reader to \cite[Section 1.4]{Braides} and \cite[Definition 4.1]{DM} for the details.

It is not difficult to see that inequalities \eqref{BGLI} and \eqref{BGLS} imply that 
\begin{equation}\label{EQinfima}
F(x)=\inf\Big\{\liminf_{n\to\infty} F_n(x_n): x_n\to x\Big\}=\inf\Big\{\limsup_{n\to\infty} F_n(x_n): x_n\to x\Big\},
\end{equation}
stating that the $\Gamma$-limit exists if and only if the two infima in \eqref{EQinfima} are equal.
Other equivalent definitions can be found, for instance, in \cite[Theorem 1.17]{Braides}; moreover, the first infimum in \eqref{EQinfima} justifies the following definition.
\begin{definition}[{\cite[Definition 1.24]{Braides}}]\label{defLGL}
Let $F_n\colon X\to\overline{\R{}}$ and let $x\in X$.
The quantity
\begin{equation}\label{eqLGL}
\Gamma-\liminf_{n\to\infty} F_n(x):=\inf\Big\{\liminf_{n\to\infty} F_n(x_n): x_n\to x\Big\}
\end{equation}
is called the \emph{$\Gamma$-lower limit} of the sequence $\{F_n\}$ at $x$.
\end{definition}
The $\Gamma$-lower limit defined in \eqref{eqLGL} is useful to treat relaxation in the framework of $\Gamma$-convergence.
Recall that the operation of relaxation is useful to treat functionals that are not lower semicontinuous - and therefore the direct method of calculus of variations cannot be applied to minimize them.
Relaxing a function means to compute its lower semicontinuous envelope.
\begin{definition}[{\cite[Definition 1.30]{Braides}}]\label{defSCE}
Let $F\colon X\to\overline{\R{}}$ be a function.
Its \emph{lower semicontinuous envelope} $\sce F$ is the greatest lower semicontinuous function not greater than $F$, that is, for every $x\in X$
\begin{equation}\label{eqSCE}
\sce F(x):=\sup\{G(x): G\text{ is lower semicontinuous and }G\leq F\}.
\end{equation}
\end{definition}
In view of \cite[Proposition 1.31]{Braides} and \cite[Remark 4.5]{DM}, relaxation is equivalent to computing the $\Gamma$-limit of a constant sequence of functions, \emph{i.e.}, $F_n=F$ for all $n$.
\begin{proposition}[{see \cite[Proposition 1.32]{Braides}}]\label{B132}
We have $\Gamma-\liminf_{n\to\infty} F_n=\Gamma-\liminf_{n\to\infty} \sce F_n$.
\end{proposition}
In view of the previous proposition, the left-hand path and the right-hand path described in the Introduction consist in the computation of two $\Gamma$-lower limits, where the order is exchanged.
For simplicity, in the paper we will use the words ``$\Gamma$-lower limit of a family of functionals'' and ``relaxation of energies'' interchangeably. 

\begin{remark}\label{PGC}
Among the properties that make $\Gamma$-convergence a suitable tool for the study of the convergence of functional and related variational problems, three are particularly useful, namely 
\begin{itemize}
\item the \emph{compactness of $\Gamma$-convergence} (see \cite[Section 1.8.2]{Braides}, \cite[Chapter 8]{DM}). 
In particular, for each sequence of functions $F_n\colon X\to\overline{\R{}}$  the compactness property grants the existence of a $\Gamma$-convergent subsequence, provided $X$ has a countable base \cite[Theorem 8.5]{DM}.
Then, it is not difficult to imagine that the choice of the topology in the convergences that define the relaxed functionals $\cF_{3d,2d}$, $\cF_{3d,2d,SD}$, $\cF_{3d,SD}$, and $\cF_{3d,SD,2d}$ below (see \eqref{107a}, \eqref{200}, \eqref{2018}, and \eqref{2019}, respectively) is made in order to obtain good compactness properties.
\item the \emph{stability under continuous perturbations} (see \cite[Remark 1.7]{Braides}, \cite[Proposition 6.21]{DM}, and also \cite[Proposition 3.7]{DM} for the relaxation): if $\widetilde F\colon X\to\R{}$ is a continuous function, then
\begin{equation}\label{CP}
\Gamma-\liminf_{n\to\infty}(F_n+\widetilde F)=\Gamma-\liminf_{n\to\infty} F_n+ \widetilde F,\quad
\Gamma-\limsup_{n\to\infty}(F_n+\widetilde F)=\Gamma-\limsup_{n\to\infty} F_n+ \widetilde F,
\end{equation}
so that if $\{F_n\}$ $\Gamma$-converges to $F$ in $X$, then $\{F_n+\widetilde F\}$ $\Gamma$-converges to $F+\widetilde F$ in $X$.
\item the implications regarding the \emph{convergence of minima and minimizers}.
The results contained in \cite[Section 1.5]{Braides} and \cite[Chapter 7]{DM}
give conditions under which the $\Gamma$-convergence of a sequence of functions $F_n$ to their $\Gamma$-limit $F$ implies the convergence of the minimima
\begin{equation}\label{CMM}
\min_{x\in X} F(x)=\lim_{n\to\infty} \inf_{x\in X} F_n(x)
\end{equation}
and of the minimizers: if $\{F_n\}$ is equi-coercive and $\Gamma$-converges to $F$, with a unique minimum point $x_0\in X$, and if $\{x_n\}\subset X$ is a sequence such that $x_n$ is an $\eps_n$-minimizer for $F_n$ in $X$ for every $n$, and with $\eps_n\to0^+$, then $x_n\to x_0$ in $X$ and $F_n(x_n)\to F(x_0)$.
We direct the reader to \cite{Braides,DM} for a precise statement of the notions of equi-coercivity and $\eps$-minimizer (albeit they are quite natural to understand).
\end{itemize}
We will not make use of the last two properties of $\Gamma$-convergence in this paper.
We think it is worthwhile mentioning them in the spirit of a variational treatment of the minimization of the relaxed functionals that we obtain in our results, with the hope to indicate to the reader that the theorems exposed and proved in the sequel can provide a starting point for the study of equilibrium configurations of thin structures in the framework of structured deformations.
\end{remark}

\begin{remark}
All the definitions and results presented above can be generalized to the case of families of functionals, indexed by a continuous parameter $\eps$.
A family of functions $\{F_\eps\}$ $\Gamma$-converges in $X$ to $F\colon X\to \overline{\R{}}$ as $\eps\to0^+$ if, for every sequence $\eps_n\to0^+$, the functions $\{F_{\eps_n}\}$ $\Gamma$-converge to $F$ in the sense of Definition \ref{GammaCVG} (see, \emph{e.g.}, \cite[Section 1.9]{Braides}).
\end{remark}


\section{The left-hand path}\label{sect:LHS}
In this section we relax our initial energy \eqref{E3d} by first doing dimension reduction and then by incorporating structured deformations.
\subsection{Dimension reduction}\label{LHSDR}
In order to perform dimension reduction, we resort to the classical approach of rescaling the spatial variable by dividing $x_3$ by $\eps$ and integrating over the rescaled domain $\Omega=\omega\times(-1/2,1/2)$. 
We also rescale the functional \eqref{E3d} by $\eps$, defining $F_{\eps}(u):=\tfrac1\eps E_{\eps}(u)$, so that we have
\begin{equation}\label{107}
F_{\eps}(u)=\int_{\Omega} W_{3d}\bigg(\nabla_{\alpha}u\bigg|\frac{\nabla_{3}u}{\eps}\bigg)\de x+\int_{\Omega\cap S(u)} h_{3d}\bigg([u],\nu_{\alpha}(u)\bigg|\frac{\nu_{3}(u)}{\eps}\bigg)\de \mathcal{H}^{2}(x). %
\end{equation}

Let now $(\overline{u},\overline{d})\in SBV(\omega;\mathbb{R}^{3})\times L^{p}(\omega;\mathbb{R}^{3})$, let $\eps_n\to0^+$, and define the relaxed functional
\begin{equation}\label{107a}
\begin{split}
\cF_{3d,2d}(\overline{u},\overline{d}):=\inf & \bigg\{\liminf_{n\to\infty} F_{\eps_n}(u_n):  u_{n}\in SBV(\Omega;\mathbb{R}^{3}),\; u_{n}\rightarrow\overline{u}\text{ in }L^{1}(\Omega;\R3), \\
& \quad\int_{I}\frac{\nabla_{3}u_{n}}{\eps_{n}}\,\de x_{3}\rightharpoonup \overline{d} \text{ in }L^{p}(\omega;\R3),\; \nu(u_{n})\cdot e_{3}=0\bigg\}.
\end{split}
\end{equation}
In writing the convergence $u_n\to\cl u$ in $L^1(\Omega;\R3)$ in formula \eqref{107a}, it is understood that $\cl u$ is extended to a function on $\Omega$ which is independent of $x_3$.
As stated in Remark \ref{energies}, the coercivity assumption \eqref{coerc} grants boundedness of the gradients in $L^p$, so that $\eps_n^{-1}\nabla_{3}u_{n}\wto d$ in $L^p$, where $d$ may still depend on the $x_3$ variable.
By contrast, following the model in \cite{BFM}, we consider the weak convergence of the average with respect to the third variable to a field $\overline d(x_\alpha)$ depending only on the coordinates in the cross--section $\omega$.

\begin{theorem}\label{T3d2d}
Under the hypotheses $(H_1)$--$(H_4)$,
let $(\cl{u},\cl{d}) \in SBV(\omega;\mathbb{R}^{3}) \times L^{p}(\omega ;\mathbb{R}^{3})$.
Then every sequence $\eps_n\to0^+$ admits a subsequence such that
\begin{equation}\label{F3d2d}
\cF_{3d,2d}(\cl{u},\cl{d}) =\int_{\omega}W_{3d,2d}(\nabla \cl{u},\cl{d})\,\de x_{\alpha}+\int_{\omega \cap S(\cl{u})}h_{3d,2d}([\cl{u}],\nu(\cl{u}))\,\de\mathcal{H}^{1}(x_\alpha),
\end{equation}
where  $W_{3d,2d}\colon \mathbb{R}^{3\times2}\times \mathbb{R}^3\to \lbrack0,+\infty)$ and $h_{3d,2d}\colon\mathbb{R}^3\times \mathbb{S}^1\to \lbrack0,+\infty)$ are given by 
\begin{equation}\label{107b}
\begin{split}
W_{3d,2d}(A,d)= & \inf\bigg\{\int_{Q^{\prime}}W_{3d}(\nabla_{\alpha}u|z)\,\de x_{\alpha}+\int_{Q^{\prime}\cap S(u)}h_{3d}([u],\tilde\nu(u))\,\de\mathcal{H}^{1}(x_\alpha):  \\
& \qquad u\in SBV(Q^{\prime};\mathbb{R}^{3}),~z\in L^p_{{Q'}-\operatorname*{per}}(\mathbb{R}^{2};\mathbb{R}^{3}),~u|_{\partial
	Q^{\prime}}(x_{\alpha})=Ax_{\alpha},  \int_{Q^{\prime}}z\,\de x_{\alpha}=d\bigg\}  ,
\end{split}
\end{equation} 
and, for $\lambda\in\R3$, $\eta\in\mathbb{S}^1$,
\begin{equation}\label{107c}
\begin{split}
h_{3d,2d}(\lambda,\eta)= &\inf\bigg\{\int_{Q^{\prime}_{\eta}\cap S(u)}h_{3d}([u],\tilde\nu(u))\; \de\mathcal{H}^{1}(x_\alpha): u\in SBV(Q_{\eta}^{\prime};\mathbb{R}^{3}), \\
&\qquad u|_{\partial Q_{\eta}^{\prime}}(x_{\alpha})=\gamma_{\lambda,\eta}(
x_{\alpha})  ,~\nabla u=0, \text{ a.e.}
\bigg\}
\end{split}
\end{equation}
with
\begin{equation}\label{gammale}
\gamma_{\lambda,\eta}(x_{\alpha}):=
\begin{cases}
\lambda & \text{if $0\leq x_{\alpha}\cdot\eta<\frac{1}{2}$},\\
0 & \text{if $-\frac{1}{2}<x_{\alpha}\cdot\eta<0$}.
\end{cases}
\end{equation}
\end{theorem}
In \eqref{107b} and in the sequel, the notation $z\in L^p_{{Q'}-\operatorname*{per}}(\mathbb{R}^{2};\mathbb{R}^{3})$ means that the function $z$ is defined on the unit square $Q'$ and extended by periodicity to all of $\R2$.

We set the stage for the proof of Theorem \ref{T3d2d} by proving some properties of the energy densities defined by \eqref{107b} and \eqref{107c}, which will be used in the sequel.
\begin{proposition}\label{p100}
Let $W_{3d, 2d}$ and $h_{3d,2d}$ be given by \eqref{107b} and \eqref{107c}, respectively. The following properties hold:
\begin{itemize}
\item [\textit{(i)}] $W_{3d,2d}$ satisfies \eqref{growth}, namely, for each $A,B\in\R{3\times2}$ and $d,e\in\R3$, 
\begin{equation}\label{wishful}
|W_{3d,2d}(A,d)-W_{3d,2d}(B,e)|\leq C|(A|d)-(B|e)|(1+|(A|d)|^{p-1}+|(B|e)|^{p-1});
\end{equation}
\item [\textit{(ii)}] $h_{3d,2d}$ satisfies $(H_2)$--$(H_4)$ and it is Lipschitz continuous with respect to the variable $\lambda$;
\item [\textit{(iii)}] $h_{3d,2d}$ is upper semicontinuous with respect to the variable $\eta$. 
\end{itemize}
\end{proposition}
\begin{proof}
\textit{(i)} Let $\Pi_\alpha\colon\R{3\times3}\to\R{3\times2}$ and  $\Pi_3\colon\R{3\times3}\to\R{3}$ be the linear maps which select out the first two and the third columns, respectively, of a matrix $M\in\R{3\times3}$. 
Note that $W_{3d,2d}(A,d)=W_{3d,2d}(\Pi_\alpha M,\Pi_3 M)$, for $M=(A|d)$.
By applying \cite[Proposition 5.6(i)]{CF} with $M\mapsto W_{3d,2d}(\Pi_\alpha M,\Pi_3 M)$ in place of $A\mapsto H_p(A,A)$, we obtain that $(A,d)\mapsto W_{3d,2d}(A,d)$ is quasiconvex (see, \emph{e.g.}, \cite[Section 2]{FM}). 
This, combined with \eqref{pgrowth}, by a standard argument by Marcellini \cite[Theorem 2.1]{M}, implies that $W_{3d,2d}$ satisfies \eqref{wishful}.

\textit{(ii)} Properties $(H_2)$--$(H_4)$ for $h_{3d, 2d}$ follow by standard arguments from \eqref{107c}. 
To prove Lipschitz continuity in the first variable, consider $\lambda_1,\lambda_2\in\R3$, $\eta\in\S1$, and $\rho>0$. 
Let now $u\in SBV(Q'_\eta;\R3)$ be admissible for $h_{3d,2d}(\lambda_1,\eta)$ in \eqref{107c} and be such that 
\begin{equation}\label{57}
h_{3d,2d}(\lambda_1,\eta)+\rho\geq\int_{Q'_\eta\cap S(u)} h_{3d}([u],\tilde\nu(u))\,\de\cH^1(x_\alpha).
\end{equation}
Then, $v := u + \gamma_{\lambda_2-\lambda_1, \eta}$ is an admissible function for the definition of $h_{3d, 2d}(\lambda_2, \eta)$ and, in view of the subadditivity of $h_{3d}$, $(H_2)$, and \eqref{57}, we have that:
\begin{equation*}
\begin{split}
h_{3d,2d}(\lambda_2,\eta) \leq \int_{Q'_\eta\cap S(v)} h_{3d}([v], \tilde\nu(v))\, \de \cH^1(x_\alpha) & \leq \int_{Q'_\eta\cap S(u)} h_{3d}([u],\tilde\nu(u))\, \de\cH^1(x_\alpha) + C|\lambda_2 - \lambda_1| \\
&\leq h_{3d, 2d}( \lambda_1, \eta) + \rho +  C|\lambda_2 - \lambda_1|.
\end{split}
\end{equation*}
Letting $\rho \to 0$ and reversing the roles of $u$ and $v$ we conclude the proof of \textit{(ii)}.

\textit{(iii)} The proof can be found in \cite[Prop.\@ 3.6]{BBBF}. \qed
\end{proof}

We prove next that, for a fixed piecewise constant $\cl d\in L^p(\omega;\R3)$, the functional $\cF_{3d,2d}^{\cl d}(\cl{u}) := \cF_{3d,2d}(\cl{u},\cl{d})$ is the trace of a Radon measure. 
To do this, we follow arguments in \cite{MS}; we start by localizing $\cF_{3d,2d}^{\cl d}(\cl{u})$, i.e., for an open set $A \in \mathcal{A}(\omega)$, $\cl u\colon\omega\to\R3$, and $\eps_n\to0^+$, we define
\begin{equation}\label{557}
\begin{split}
\cF_{3d,2d}^{\overline{d}}(\overline{u};A):=\inf & \bigg\{\liminf_{n\to\infty}\bigg(\!\int_{A \times I}W_{3d}\Big(\nabla_{\alpha}u_{n}\Big|\frac{\nabla_{3}u_{n}}{\eps_{n}}\Big)\de x+\!\int_{(A \times I)\cap S(u_{n})} \!\! h_{3d}\Big([u_{n}],\nu_{\alpha}(u_{n})\Big|\frac{\nu_{3}(u_{n})}{\eps_{n}}\Big)\de\mathcal{H}^{2}(x)\bigg):  \\
& \quad  u_{n}\in SBV(\Omega;\mathbb{R}^{3}),\; u_{n}\rightarrow\overline{u}\text{ in }L^{1},\; \int_{I}\frac{\nabla_{3}u_{n}}{\eps_{n}}\,\de x_{3}\rightharpoonup \overline{d} \text{ in }L^{p},\; \nu(u_{n})\cdot e_{3}=0\bigg\}.
\end{split}
\end{equation}
Notice that the functional defined in \eqref{557} depends on the particular sequence $\{\eps_n\}$ (but for simplicity we do not write it explicitly).
Then we have the following result.
\begin{proposition}\label{569}
Let $W_{3d}\colon{\mathbb{R}}^{3\times3}\rightarrow\lbrack0,+\infty)$ and
$h_{3d}\colon{\mathbb{R}}^{3}\times \S{2}\rightarrow\lbrack0,+\infty)$ be continuous satisfying $(H_1)$ and $(H_2)$ and let $\cl d \in L^p(\omega; \R3)$ be piecewise constant. Any sequence $\eps_n \to 0^+$ admits a subsequence $\eps_k: = \eps_{n(k)}$ such that for $\cl u \in SBV(\omega; \R3)$ the set function $\cF_{3d,2d}^{\overline{d}}(\overline{u};\cdot)$ defined in \eqref{557}, is the trace of a Radon measure on $\cA(\omega)$ which is absolutely continuous with respect to $\mathcal{L}^2 + \mathcal{H}^1\res {S(\cl u)}.$
\end{proposition}
\begin{proof}
We start by noting that, considering the admissible sequence $u_n := \cl u + \eps_n x_3 \cl d$, by $(H_1)$ and $(H_2)$ the following upper bound holds
$$\cF_{3d,2d}^{\cl d}(\cl u;A) \leq C\bigg( \mathcal{L}^2 (A) + \int_A |\nabla \cl u|^p \de x_\alpha + \int_A |\cl d|^p \de x_\alpha + \|D \cl u \|(\overline{A})\bigg).$$ 
For each $a\in\omega$ with rational coordinates and for $i\in\N{}$, consider balls $B(a;r_i)$ with radii $r_i$ and depending on $a$, such that
\begin{equation}\label{577} \left | r_i - \frac{1}{i}\right| \leq \frac{1}{i^2},\; \; \; \; \overline{B(a; r_i)} \subset \omega,\; \;  \; \; \|D^s \cl u\|(\partial B(a; r_i)) = 0.\end{equation}
Let $\mathcal{B}(\omega)$ be the set of all such balls and their finite unions. The set of all closed balls $\overline{B(a; r_i)} $ is a fine cover of $\omega$ (see \cite[p.\@ 49]{AFP}).
Given a sequence $\eps_n \to 0^+$, by a standard diagonalization argument, we can take an appropriate subsequence $\eps_k: = \eps_{n(k)}$ such that, for each $B \in \mathcal{B}(\omega)$, we may find a sequence $u_k$ (depending on $B$) such that
\begin{equation}\label{580} u_{k}\rightarrow\overline{u}\text{ in }L^{1},\quad\int_{ I}\frac{\nabla_{3}u_{k}}{\eps_{k}}\,\de x_{3}\rightharpoonup \cl{d} \text{ in }L^{p},\quad \nu(u_{k})  \cdot
e_{3}=0, \end{equation}
and
\begin{equation}\label{585}
\cF_{3d,2d}^{\cl d}(\cl u;B) = \lim_{k\to \infty}\bigg(\int_{B \times I}W_{3d}\Big(\nabla_{\alpha}u_{k}\Big|\frac{\nabla_{3}u_{k}}{\eps_{k}}\Big)\de x+\int_{(B \times I)\cap S(u_{k})}h_{3d}\Big([u_{k}],\nu_{\alpha}(u_{k})\Big|\frac{\nu_{3}(u_{k})}{\eps_{k}}\Big)\de\mathcal{H}^{2}(x)\bigg).
\end{equation}
Next, we prove the following subadditivity property: for every $B, B_1, B_2 \in \mathcal{A}(\omega)$ such that $ B_1 \Subset  B \Subset B_2$, we have that
\begin{equation}\label{591}
\cF_{3d,2d}^{\cl d}(\cl u;B_2) \leq \cF_{3d,2d}^{\cl d}(\cl u;B) + \cF_{3d,2d}^{\cl d}(\cl u;B_2 \backslash \overline{B}_1).
\end{equation}
To this end, for each $B \in \mathcal{A}(\omega)$, define the Radon measure
$$\Delta(B) := C \bigg(\mathcal{L}^2(B) + \int_B |\cl d|^p \de x_\alpha + \int_B |\nabla \cl u|^p  \de x_\alpha + \|D^s \cl u\|(B)\bigg).$$
For fixed $\rho >0$ consider an open set $B_\rho \in \mathcal{B}(\omega)$ such that $B_\rho \subset B$. 
Using the Besicovitch covering theorem, we can find $A_\rho \in \mathcal B(\omega)$ such that $A_\rho \subset B_2 \backslash \overline{B}_1$ and
$$\Delta \left( (B_2\setminus \overline{B}_1)\setminus \overline{A}_\rho \right) < \rho.$$
Note that we can choose the sets above such that  there exists an open set $\tilde{A}$ with Lipschitz boundary and with $B_1\Subset\tilde{A} \Subset  B_\rho$ and with $\partial \tilde{A} \subset A_\rho$. Now, consider $\{u_k^1\} \in SBV(A_\rho; \R3)$ and $\{u_k^2\} \in SBV(B_\rho; \R3)$ satisfying \eqref{580} and \eqref{585}, and define
$$\tilde{u}_k := \left\{ \begin{array}{ll} u_k^1 & \text{in}\; A_\rho \setminus \tilde{A}\\
u_k^2 & \text{in}\; \tilde{A}\\
u_k & \text{otherwise in $B_2$,}
\end{array}
\right.
$$
where $u_k(x_\alpha, x_3):= \overline{u}(x_\alpha) + \eps_k x_3 \overline{d}(x_\alpha)$.
Notice that $\tilde{u}_k \in SBV(B_2; \R3)$ by \cite[Proposition 3.21]{AFP}.
Then we have that
\begin{align*} 
\cF_{3d,2d}^{\overline{d}}(\overline{u};B_2)&\leq \lim_{k\to\infty} \bigg( \int_{B_2 \times I} W_{3d}\Big(\nabla_\alpha \tilde{u}_k \Big| \frac{\nabla_3 \tilde{u}_k}{\eps_k}\Big)\de x + \int_{(B_2 \times I)\cap S(\tilde{u}_k)} h_{3d}([\tilde{u}_k], \tilde\nu_\alpha(\tilde{u}_k))\,\de \mathcal{H}^2(x)\bigg)\\
& \leq \cF_{3d,2d}^{\overline{d}}(\overline{u};A_\rho) + \cF_{3d,2d}^{\overline{d}}(\overline{u};B_\rho) + \Delta\left((B_2\setminus \overline{B}_1)\setminus \overline{A}_\rho \right)\\
&\leq \cF_{3d,2d}^{\overline{d}}(\overline{u};A_\rho) + \cF_{3d,2d}^{\overline{d}}(\overline{u};B_\rho)+\rho\\
&\leq \cF_{3d,2d}^{\overline{d}}(\overline{u};B_2\setminus \overline{B}_1) + \cF_{3d,2d}^{\overline{d}}(\overline{u};B)+ \rho. 
\end{align*}
Note that, since $A_\rho \setminus \tilde{A} \Subset A_\rho, \tilde{A} \Subset B_\rho$ and by \eqref{577}, the jumps of $\tilde{u}_k$ in the transition layers are included in the computations above. By letting $\rho \to 0$, we have that \eqref{591} holds.

In the following, let $u_k = u_k^{\omega}$ denote an appropriate sequence for which \eqref{585} holds in $\omega$. 
Define the sequence of bounded Radon measures
$$\Lambda_k(A) := \int_{A\times I} W_{3d}\Big(\nabla_\alpha u_k \Big| \frac{\nabla_3 u_k}{\eps_k}\Big) \; dx + \int_{(A \times I)\cap S(u_k)} h_{3d}([u_k], \tilde\nu_\alpha (u_k)) \; \de \mathcal{H}^2(x),$$
for $A \in \mathcal{A}(\omega)$ and extract a subsequence (not relabeled) such that $\Lambda_k \wsto \Lambda$. In order to complete the proof we show that for every $A \in \mathcal{A}(\omega)$ we have that
\begin{equation}\label{616}
\cF_{3d,2d}^{\overline{d}}(\overline{u};A) = \Lambda (A).
\end{equation}
Note first that for any $A \in \mathcal{A}(\omega)$, open set, the following inequality holds
\begin{equation}\label{621}
\cF_{3d,2d}^{\overline{d}}(\overline{u};A) \leq \Lambda (\overline{A}).
\end{equation}
Given $B \in \mathcal{A}(\omega)$, let $\rho > 0$ and consider $W \Subset B$ such that $\Lambda(B \backslash W) < \rho.$ Then, since $\Lambda (\omega) = \Lambda(\overline{\omega}),$ by \eqref{591} and \eqref{621} we have that
\begin{align*}
\Lambda (B) &\leq \Lambda( W) + \rho\\
& = \Lambda (\omega) - \Lambda(\omega \setminus W) + \rho\\
& \leq \cF_{3d,2d}^{\overline{d}}(\overline{u};\omega) - \cF_{3d,2d}^{\overline{d}}(\overline{u};\omega \setminus \overline{W}) + \rho\\
& \leq \cF_{3d,2d}^{\overline{d}}(\overline{u};B) + \rho.
\end{align*}
Letting $\rho \to 0$ we have that
\begin{equation}\label{632}
\Lambda (B) \leq \cF_{3d,2d}^{\overline{d}}(\overline{u};B).
\end{equation}
Finally, it remains to prove the reverse inequality. Let now $K \subset B$ be a compact set such that $\Delta(B\setminus K) < \rho$ and choose an open set $D$ such that $ K \Subset D \Subset B$. Again, by \eqref{591} we have
\begin{align*} 
\cF_{3d,2d}^{\overline{d}}(\overline{u};B) & \leq \cF_{3d,2d}^{\overline{d}}(\overline{u};D) + \cF_{3d,2d}^{\overline{d}}(\overline{u};B \setminus K)\\
& \leq \Lambda (\overline{D}) + \Delta (B\setminus K)\\
& \leq \Lambda (B) + \rho,
\end{align*}
which, together with \eqref{632} and letting $\rho \to 0$, yields the result. 
\qed
\end{proof}

Notice that, by Proposition \ref{569}, for every sequence $\{\eps_n\}$, the localized functional $\cF_{3d,2d}^{\overline{d}}(\overline{u};\cdot)$ defined in \eqref{557} is the trace of a Radon measure on $\cA(\omega)$, so that it admits an integral representation.

We are now ready to prove Theorem \ref{T3d2d}, and we divide its proof into four steps, each of which relies on the blow-up method of \cite{FM}:
we will prove upper bounds for the Radon-Nikod\'{y}m derivatives of $\cF_{3d,2d}(\cl u,\cl d)$ with respect to $\cL^2$ and $\cH^1\res S(\cl u)$ at a point $x_0\in\omega$ (see \eqref{1151} and \eqref{uppersurf}, respectively), and lower bounds for the Radon-Nikod\'{y}m derivatives of a certain measure $\mu$ (the weak-* limit of the measures $\mu_n$ defined in \eqref{miu_n}) with respect to $\cL^2$ and $\vert [\cl u] \vert\mathcal{H}^{1}\res S(\cl u)$ (see \eqref{778} and \eqref{779}, respectively).
We will find that these upper and lower bounds are indeed independent of the particular choice of the sequence $\eps_n\to0^+$, so that estimates \eqref{1151}, \eqref{uppersurf}, \eqref{778}, and \eqref{779} will suffice to conclude the proof of the theorem.

Moreover, we note the connection with the theory of $\Gamma$-convergence presented in Section \ref{Gc}: Steps 1 and 2 correspond to proving the $\limsup$ inequality \eqref{BGLS}, Steps 3 and 4 correspond to proving the $\liminf$ inequality \eqref{BGLI}.
\paragraph{\textit{\textbf{Step 1} (Upper bound -- bulk)}}
We start by noticing that, by Lemma \ref{ctap} and \eqref{growth},  it is enough to derive the upper bound for the case where $\overline{d}$ is piecewise constant. 
In fact, given $u_n$ admissible for $\cF_{3d,2d}(\cl u, \cl d)$ and $\cl d_k$ a piecewise constant approximation of $\cl d$ given by Lemma \ref{ctap}, for each $k$ we can obtain an admissible sequence $u_{k,n}$ for $\cF_{3d,2d}(\cl u,\cl d_k)$ by defining $u_{k,n}:=u_n+h_{k,n}$, where $h_{k,n}$ is provided by Theorem \ref{Al} in such a way that $\nabla h_{k,n} =\eps_n \Big(0 \Big| \cl d_k - \int_I \frac{\nabla_3 u_n}{\eps_n}\de x_3\Big)$ and $||h_{k,n}||_{L^1(\Omega;\R3)}\leq C\eps_n\big(||\cl d_k||_{L^p(\omega;\R3)}+||\cl d||_{L^p(\omega;\R3)}\big)$. 
Therefore,
$$ \cF_{3d,2d}(\cl u, \cl d) \leq \liminf_{k \to \infty}\cF_{3d,2d}( \cl u, d_k) \leq \limsup_{k \to \infty} \bigg(\int_\omega W_{3d, 2d}(\nabla \cl u, d_k)\, \de x_\alpha + \int_{\omega \cap S(\cl u)} h_{3d, 2d}([\cl u], \nu (\cl u))\, \de \mathcal{H}^1(x_\alpha)\bigg),$$
and the result follows because $W_{3d, 2d}$ has growth of order $p$ (see Proposition \ref{p100}). 
Let $(\overline{u},\overline{d})\in SBV(\omega;\mathbb{R}^3)\times L^{p}(\omega;\mathbb{R}^{3})$, with $\cl d$ piecewise constant, and let $x_{0}\in\omega$ be chosen such that
\begin{equation}\label{apcontu}
\lim_{\delta\to0}\frac{1}{\delta^{2}} |D^s \overline{u}|(Q'(x_0, \delta)) = 0,
\end{equation}
\begin{equation}\label{apcont}
\lim_{\delta\to0}\frac{1}{\delta^{2}}\int_{Q'(x_{0},\delta)}\left\vert
\overline{d}(x_\alpha)-\overline{d}(x_{0})\right\vert^p \de x_\alpha=0, 
\end{equation}
\begin{equation}
\lim_{\delta\to0}\frac{1}{\delta^{2}}\int_{Q'(x_{0},\delta)}\left\vert \nabla_\alpha \cl u(x_\alpha) - \nabla_\alpha \cl u(x_0))\right\vert^p \,\de x_\alpha=0. \label{apdiff}
\end{equation}
It suffices to prove
\begin{equation}\label{1151}
\frac{\de\cF_{3d,2d}(\overline{u},\overline{d})}{\de\mathcal{L}^{2}}(x_{0})\leq W_{3d,2d}(\nabla_{\alpha}\overline{u}(x_{0}),\overline{d}(x_{0})).
\end{equation}
To this end, fix $\rho>0$ and choose $u\in SBV(Q';\mathbb{R}^{3})$ and $z\in L_{Q'-\operatorname*{per}}^{p}(\mathbb{R}^{2};\mathbb{R}^{3})$ piecewise constant such that
\begin{equation}
u|_{\partial Q'}(x_{\alpha})=\nabla_{\alpha}\overline{u}(x_{0})x_{\alpha},\qquad\int_{Q'}z(x_{\alpha})\,\de x_{\alpha}=\overline{d}(x_{0}),
\label{115}
\end{equation}
and
\begin{equation}
W_{3d,2d}(\nabla_{\alpha}\overline{u}(x_{0}),\overline{d}(x_{0}))+\rho\geq\int_{Q'}W_{3d}(\nabla_{\alpha}u|z)\,\de x_{\alpha}+\int_{Q'\cap S_{u}}h_{3d}([u],\tilde{\nu})\,\de\mathcal{H}^{1}(x_\alpha). \label{114}
\end{equation}

We now construct a sequence ${u_{\delta,n}}$ of competitors for the problem \eqref{107a} by setting $\zeta(x_{\alpha}):=u(x_{\alpha})-\nabla_{\alpha}\overline{u}(x_{0})x_{\alpha}$ (extended by periodicity to all of $\R2$) and defining
\begin{equation}
u_{\delta,n}(x_{\alpha},x_{3}):=\overline{u}(x_{\alpha})+\frac{\delta}{n}\zeta\Big(\frac{n(x_{\alpha}-x_{0})}{\delta}\Big)+\eps_{n}x_{3}\Big(\overline{d}(x_{\alpha})-\overline{d}(x_{0})+z\Big(\frac{n(x_{\alpha}-x_{0})}{\delta}\Big)\Big). 
\label{116}
\end{equation}
Clearly, $u_{\delta,n} \in SBV(\Omega;\mathbb{R}^{3})$, $\lim_{\delta,n} u_{\delta, n}=\overline{u}$ in $L^1(\Omega;\R3)$, and
\begin{equation}\label{119a}
\int_{I}\frac{\nabla_{3}u_{n}(x_{\alpha},x_{3})}{\eps_{n}}\,\de x_{3}=\overline{d}(x_{\alpha})-\overline{d}(x_{0})+z\Big(\frac{n(x_{\alpha}-x_{0})}{\delta}\Big).
\end{equation}
It is not difficult to see that $z(n(x_\alpha-x_0)/\delta)\wto\int_{Q'} z(x_\alpha)\,\de x_\alpha$ in $L^p(\omega;\R3)$, so that, by \eqref{115}, the right-hand side of \eqref{119a} converges to $\cl d(x_\alpha)$ as $n\to\infty$.
Notice that in the construction of ${u_{\delta, n}}$ the normal $\nu(u_{\delta,n})$ satisfies $\nu(u_{\delta,n})\cdot e_3 = 0$. 

Since $\overline{d}$ and $z$ are piecewise constant, we have
\begin{equation}\label{121}%
\nabla_{\alpha}u_{\delta, n}(x_{\alpha},x_{3})=\nabla_{\alpha}\overline{u}(x_{\alpha})+\nabla_{\alpha}\zeta\Big(\frac{n(x_{\alpha}-x_{0})}{\delta}\Big)=\nabla_{\alpha}\overline{u}(x_{\alpha})+\nabla_{\alpha}u\Big(\frac{n(x_{\alpha}-x_{0})}{\delta}\Big)-\nabla_{\alpha}\overline{u}(x_{0}). 
\end{equation}
Therefore, recalling $(H_1)$--$(H_4)$,
\begin{equation}
\begin{split}
\frac{\de\cF_{3d,2d}(\overline{u},\overline{d})}{\de\mathcal{L}^{2}}(x_{0})\leq  \lim_{\delta,n}\frac{1}{\delta^{2}} \bigg\{& \int_{Q'(x_{0};\delta){\times}I} W_{3d}\Big(\nabla_{\alpha}u_{\delta,n}\Big| \frac{\nabla_{3}u_{\delta,n}}{\eps_{n}}\Big)\de x  \\ 
& +\int_{(Q'(x_{0};\delta){\times}I)\cap S(u_{\delta,n})} h_{3d}([u_{\delta,n}],\tilde\nu_{\alpha}(u_{\delta,n}))\,\de\cH^{2}(x) \bigg\}  \\
 \leq   \lim_{\delta,n}\frac{1}{\delta^{2}}\bigg\{ & \int_{Q'(x_{0};\delta)}W_{3d}\Big( \nabla_{\alpha}u\Big(\frac{n(x_{\alpha}-x_{0})}{\delta}\Big)  \Big| z\Big(\frac{n(x_{\alpha}-x_{0})}{\delta}\Big)\Big)  \de x_{\alpha} \label{122}\\
& +  \int_{Q'(x_{0};\delta)\cap(x_{0}+\frac{\delta}{n}S(u))} h_{3d}\Big(\frac{\delta}{n}[u]\Big(\frac{n(x_{\alpha}-x_{0})}{\delta}\Big), \tilde\nu_{\alpha}(u)\Big) \de\mathcal{H}^{1}(x_{\alpha})\bigg\}  \\
 +\lim_{\delta\to0}
\frac{1}{\delta^{2}}& \int_{Q'(x_{0};\delta)}\left[|\nabla_{\alpha}\overline{u}(x_{\alpha})-\nabla_{\alpha}\overline{u}(x_{0})|^p+|\overline{d}(x_{\alpha})-\overline
{d}(x_{0})|^p\right]\,\de x_{\alpha}  \\
 +\lim_{\delta,n}\frac{1}{\delta^{2}} & |D^s\overline{u}|(Q'(x_0; \delta)) \\
+\lim_{\delta\to0}\frac1{\delta^2} & \limsup_{n\to\infty}{\eps_{n}}\bigg\{\int_{(Q'(x_{0};\delta){\times}I)\cap S(\cl{d})}|x_{3}(\overline{d}(x_{\alpha})-\overline{d}(x_{0}))|\,\de\mathcal{H}^{1}(x_{\alpha})\de\mathcal{L}^{1}(x_3)  \\
&+\int_{(Q' \times I)\cap S(z)}\frac{\delta}{n}|x_{3}z(y_{\alpha})|\,\de\mathcal{H}^{1}(y_{\alpha})\de\mathcal{L}^{1}(x_3)\bigg\}, 
\end{split}
\end{equation}
where in the last integral we performed the change of variables $y_{\alpha
}:=n(x_{\alpha}-x_{0})/\delta$. By the same change of variables and noticing that, by \eqref{apcontu}, \eqref{apcont}, \eqref{apdiff}, and the hypothesis on $z$, the last four terms in \eqref{122} vanish, we are left with
\begin{equation}
\begin{split}
\frac{\de\cF_{3d,2d}(\overline{u},\overline{d})}{\de\mathcal{L}^{2}}(x_{0}) 
\leq &  \lim_{\delta,n}\frac{1}{n^{2}}\bigg\{  \int_{nQ'}W_{3d}(\nabla_{\alpha}u(y_{\alpha})|z(y_{\alpha}))\,\de y_{\alpha}+\int_{nQ'\cap S(u)}
h_{3d}([u](y_{\alpha}),\tilde{\nu}(u))\,\de\mathcal{H}^{1}(y_{\alpha})\bigg\} \\
\leq  &  \int_{Q'}W_{3d}(\nabla_{\alpha}u(y_{\alpha})|z(y_{\alpha}))\,\de y_{\alpha}+\int_{Q'\cap S(u)}h_{3d}([u](y_{\alpha}),\tilde{\nu}(u))\,\de\mathcal{H}^{1}(y_{\alpha})\\
\leq &  W_{3d,2d}(\nabla_{\alpha}\overline{u}(x_{0}),\overline{d}(x_{0}))+\rho,
\end{split}
\label{123}
\end{equation}
where we have used the periodicity of the functions $z$ and $\zeta$, assumption $(H_2)$, and \eqref{114}. 
The arbitrary choice of $\rho$ yields now \eqref{1151}.
By approximating with piecewise constant functions (see Lemma \ref{ctap}) and using \eqref{growth}, the estimate is extended to a general $z$.

\paragraph{\textit{\textbf{Step 2} (Upper bound -- surface)}}

Following an argument in \cite{AMT}, and taking into account Proposition \ref{p100},  it suffices to prove the upper bound for the case where $\cl u$ is of the form  $\cl u = \lambda \chi_U$, where $\chi_U$ denotes the characteristic function of a set of finite perimeter $U\subset\omega$ and $\lambda\in\R3$. Moreover, by standard arguments we can restrict ourselves to the case where $U$ is a polygonal set.
Given $x_0 \in S(\cl u)$, writing for simplicity $\nu := \nu (\cl u)(x_0)$, by the definition of
$h_{3d,2d}$, for any $\rho>0$ we may find $u\in SBV(Q'_{\nu};\mathbb{R}^{3})$, such that $\nabla_\alpha u=0$ a.e., $u|_{\partial Q_\nu'}=\gamma_{\lambda,\nu}$, and
$$\int_{Q'_{\nu}\cap S(u)}h_{3d}([u],\tilde{\nu})\,\de\mathcal{H}^{1}(x_\alpha)\leq h_{3d,2d}([\overline{u}],\nu(\overline{u}))(x_{0})+\rho.$$

We claim that
\begin{equation}\label{uppersurf}
\frac{\de\cF_{3d,2d}(\overline{u},\overline{d})}{\de\mathcal{H}^{1}\res S(\overline{u})}(x_{0})  \leq h_{3d,2d}([\overline{u}](x_{0}),\nu(\cl u)(x_{0})),
\end{equation}
for $\mathcal{H}^1-$ a.e. $x_0 \in \omega\cap S(\overline{u}).$
Now put $\lambda:=[\overline{u}](x_{0})$, and since it is not restrictive to assume that $\nu =e_{2}$, define 
\begin{align*}
	D_{n}(x_{0},\delta)   & := \bigg( Q'(x_{0},\delta)  \cap\bigg\{ x:\bigg\vert (x-x_{0})\cdot e_2\bigg\vert <\frac{\delta}{2n}\bigg\}\bigg)\times I,\\
	Q^{+}(x_{0},\delta)   & := \left (Q'(x_{0},\delta)  \cap \left\{ x:(x-x_{0})  \cdot e_2>0\right\}\right) \times I,\\
	Q^{-}(x_{0},\delta)   & := \left (Q'(x_{0},\delta)  \cap \left\{  x: (x-x_{0})  \cdot e_2<0\right\}\right)\times I.
\end{align*}
Let
\[u_{\delta,n}(x_{\alpha},x_{3}):=\left\{\begin{array}[c]{lll}
\lambda+\eps_{n}x_{3}\overline{d}, &  & \text{in } Q^{+}(x_{0},\delta)  \backslash D_{n}(x_{0},\delta),\\
\displaystyle u\left(\frac{n(x_{\alpha}-x_{0})}{\delta}\right)   &  & \text{in } D_{n}(x_{0},\delta),\\
\eps_{n}x_{3}\overline{d} &  & \text{in }Q^{-}(x_{0},\delta)  \backslash D_{n}(x_{0},\delta).
\end{array}
\right.
\]

Clearly, $u_{\delta,n}\rightarrow \cl u$ in $L^{1}(Q(x_{0},\delta) ;\mathbb{R}^{3})$ (that is, it converges to $\tilde{u}(x_\alpha, x_3) := \cl u(x_\alpha)$),
$\tfrac1{\eps_{n}}\int_I\nabla_{3}u_{\delta,n}\,\de x_{3}\rightharpoonup
\overline{d}$ in $L^{p}(Q(x_{0},\delta) ;\mathbb{R}^{3})$, both as $n\to\infty$, and $\nu(u_{\delta,n})\cdot e_{3}=0.$ 

Thus, 
\begin{align*}
	\frac{\de\cF_{3d,2d}(\overline{u},\overline{d})}{\de\mathcal{H}^{1}\res S(\cl u)}(x_{0})   \leq\lim_{\delta, n}
	\frac{1}{\delta}\bigg\{ &  \int_{Q'(x_{0},\delta)\times I}
	W_{3d}\left(\nabla_{\alpha}u_{\delta,n}\left|\frac{\nabla_{3}u_{\delta,n}}{\eps_{n}}\right.\right)\de x\\
	& +\int_{(Q'(x_{0},\delta) \times I)\cap S(u_{\delta,n})}h_{3d}\left([u_{\delta,n}],\nu_{\alpha}(u_{\delta, n}) \left|\frac{\nu_{3}(u_{\delta,n})}{\eps_{n}}\right.\right)\de\mathcal{H}^{2}(x)\bigg\} \\
 =\lim_{\delta,n}\frac{1}{\delta}\bigg\{ &\int_{(Q'(x_{0},\delta) \times I) \setminus D_{n}(x_{0,}\delta) }W_{3d}(0|\overline{d})\,\de x \\
 &+\int_{[(Q'(x_{0},\delta) \times I) \setminus D_{n}(x_{0,}\delta)] \cap (S(\cl d)\times I)}h_{3d}\left(\eps_n x_3[\cl d],\tilde{\nu}(\cl d)\right)\de\mathcal{H}^{2}(x)  \\
	&  +\int_{D_{n}(x_{0,}\delta) }W_{3d}\left(\frac{n}{\delta}\nabla_{\alpha}u\left.\left( \frac{n(x_{\alpha}-x_{0})}{\delta}\right) \right|0\right)\de x \\
	& +\int_{D_{n}(x_{0,}\delta) \cap \{ x_0 + \frac{\delta}{n} S(u)\}\times I} h_{3d}\left([u]\left(  \frac{n\left(  x_{\alpha}
		-x_{0}\right)}{\delta}\right),\tilde{\nu}(u)\right)\de\mathcal{H}^{2}(x)\bigg\}.  \\
\end{align*}
Using now the growth conditions on $W_{3d}$ and $h_{3d}$ and changing variables one obtains
\begin{equation*}
\begin{split}
\frac{\de\cF_{3d,2d}(\overline{u},\overline{d})}{\de\mathcal{H}^{1}\res S(\cl u)}(x_{0}) \leq  \lim_{\delta,n}\frac{1}{\delta}\bigg\{ & \int_{Q^{\prime}(x_{0},\delta )}C(1+|\cl d|^{p}) \,\de x_\alpha  + c_h\, \eps_n |D^s \cl d|(Q'(x_0, \delta))\\
& +\int_{D_{n}(x_{0},\delta)}W_{3d}\Big(\frac{n}{\delta }\nabla_{\alpha }u\Big( \frac{n(x_{\alpha }-x_{0})}{\delta }\Big)\Big|0\Big) \de x \\
& +\int_{D_{n}(x_{0},\delta)\cap \{ x_0 + \frac{\delta}{n} S(u)\}\times I}h_{3d}\Big ([u]\Big(\frac{n(x_{\alpha }-x_{0})}{\delta}\Big) ,\tilde{\nu}(u)\Big) \de\mathcal{H}^{2}(x) \bigg\} \\
\leq \lim_{\delta ,n}\bigg\{ & \frac{\delta}{n^2}\int_{nQ'\times I}W_{3d}\Big(\frac{n}{\delta }\nabla_{\alpha }u(y_{\alpha})\Big|0\Big)\de y \\
& +\frac{1}{n}\int_{(nQ'\times I)\cap (S(u)\times I) \cap \{ y\cdot e_2| \leq \frac{1}{2}\}}h_{3d}([u](y_{\alpha}),\tilde{\nu}(u))\de \mathcal{H}^{2}(y)\bigg\},
\end{split}
\end{equation*}
since, without loss of generality, the piecewise constant function $\cl d$ can be taken to belong to $L^\infty$ (see the proof of Lemma \ref{ctap}).
Moreover, since $\nabla_\alpha u = 0$, we have that:
$$\lim_{\delta ,n} \frac{\delta}{n^2}\int_{nQ'\times I}W_{3d}\Big(\frac{n}{\delta }\nabla_{\alpha }u(y_{\alpha })\Big|0\Big) \de y \leq C\delta,$$
and this term also vanishes in the limit $\delta\to0$.
We then have that
\begin{equation*} 
\begin{split}
\frac{\de\cF_{3d,2d}(\overline{u},\overline{d})}{\de\mathcal{H}^{1}\res S(\cl u)}(x_{0}) \leq & \lim_{\delta, n} \frac{1}{n}\int_{(nQ' \times I)\cap (S(u)\times I) \cap \{ y\cdot e_2| \leq \frac{1}{2}\}}h_{3d}([u](y_{\alpha }),\tilde{\nu }(u))\de \mathcal{H}^{2}(y)\\
\leq &\liminf_{n \to \infty}\frac{1}{n} \int_{(nQ'\times I) \cap (S(u)\times I) \cap \{ y\cdot e_2| \leq \frac{1}{2}\}}h_{3d}([u](y_{\alpha }),\tilde{\nu}(u))\de \mathcal{H}^{2}(y)  \\
=&\int_{Q'\cap S(u)} h_{3d} ([u], \tilde\nu(u)) \, \de \mathcal{H}^1 \leq h_{3d, 2d}(\lambda, \nu) + \rho,
\end{split}
\end{equation*}
from which \eqref{uppersurf} follows.

\paragraph{\textit{\textbf{Step 3} (Lower bound -- bulk)}}
Given a set $B \in \cA(\omega)$, let $u_n\in SBV(\Omega;\R3)$ be an admissible sequence for $\cF_{3d,2d}(\overline{u},\overline{d})(B)$ with $\mu_n$ the corresponding sequence of nonnegative Radon measures given by 
\begin{equation}\label{miu_n}
\mu_{n}(B)  :=\int_{B\times I}W_{3d}\Big(\nabla_{\alpha}u_{n}\Big|\frac{\nabla_{3}u_{n}}{\eps_{n}}\Big)\de x+\int_{(B\times I)\cap S(u_{n})} h_{3d} ([u_{n}],\tilde\nu(u_{n})
) \de\mathcal{H}^{2}(x).
\end{equation}
Let $x_{0}\in\omega\,$ satisfying 
\begin{equation}\label{827}
\lim_{\delta \to 0} \frac{1}{\delta^{3}}\int_{Q'(x_0, \delta)} | \cl u(x_\alpha) - \cl u (x_0) - \nabla \cl u (x_0) (x_0 - x_\alpha)| \, \de x_\alpha = 0,
\end{equation}
\begin{equation}\label{830}
\lim_{\delta \to 0} \frac{1}{\delta^2} \int_{Q'(x_0, \delta)}|\cl d(x_\alpha) - \cl d(x_0)|^p\, \de x_\alpha = 0.
\end{equation}
By $(H_1)$ and $(H_2)$ $\mu_{n}$ is bounded and so, up to subsequence (not relabeled), there exists a positive Radon measure $\mu$ such that $\mu_{n}\overset{\ast}{\rightharpoonup}\mu$. 
In addition, choose $x_0\in\omega$ such that $\frac{\de\mu}{\de\mathcal{L}^{2}}(x_{0})$ exists and is finite. 
Moreover, there exists a sequence of radii ${\delta_{k}}\to0$ such that $\mu(\partial Q(x_{0},\delta_{k}))=0$ for every $k\in\mathbb{N}.$

It suffices to prove that
\begin{equation}\label{778}
\frac{\de\mu}{\de\mathcal{L}^{2}}(x_0)  \geq W_{3d,2d}(\nabla_{\alpha}\overline{u}(x_0)  ,\overline{d}(x_0))\qquad\text{for $\mathcal{L}^2-$ a.e. $x_0 \in \omega.$}
\end{equation}
We have
\begin{equation*}
\begin{split}
\frac{\de\mu}{\de\mathcal{L}^{2}}(x_{0})   & = \lim_{k,n}\frac{1}{\delta_{k}^{2}}\mu_{n}(Q(x_{0},\delta_{k})) \\
&=\lim_{k,n}\frac{1}{\delta_{k}^{2}} \bigg(\int_{Q'(x_{0},\delta_{k})  \times I}W_{3d}\Big(\nabla_{\alpha}u_{n} \Big|\frac{\nabla_{3}u_{n}}{\eps_{n}}\Big)\de x +\int_{(Q'(x_{0},\delta_{k})  \times I)\cap S(u_{n})}h_{3d}([u_{n}],\tilde\nu(u_{n}))\,\de\mathcal{H}^{2}(x)\bigg).
\end{split}
\end{equation*}
Performing the change of variables $y_{\alpha}=(x_{\alpha}-x_{0})/\delta_{k}$ one obtains
\begin{align*}
	\frac{\de\mu}{\de\mathcal{L}^{2}}(x_{0})   & =\lim_{k,n}\bigg\{ \int_{Q'\times I}W_{3d}\Big(\nabla_{\alpha}u_{n}(x_{0}+\delta_{k}y_{\alpha},y_{3})\Big|\frac{\nabla_{3}u_{n}(x_{0}+\delta_{k}y_{\alpha},y_{3})}{\eps_{n}}\Big)  \de y \\
	& +\frac{1}{\delta_{k}} \! \int_{(Q'\times I)\cap\{(y_{\alpha},y_{3}):(x_{0}+\delta_{k}y_{\alpha},y_{3})  \in S(u_{n}) \}} \!\!\!\! h_{3d}([u_{n}](x_{0}+\delta_{k}y_{\alpha},y_{3}),\tilde\nu_{\alpha}(u_{n})  (x_{0}+\delta_{k}y_{\alpha},y_{3}))\,\de\mathcal{H}^{1}(y_{\alpha})\de y_{3}\bigg\}  .
\end{align*}

Defining
$$u_{k,n}(y):=\frac{u_{n}(x_{0}+\delta_{k}y_{\alpha},y_{3})-\overline{u}(x_{0})}{\delta_{k}},$$
we have
\begin{equation*}
\nabla_{\alpha}u_{k,n}(y)  =\nabla_{\alpha}u_{n}(x_0+\delta_ky_\alpha,y_3),\quad \nabla_{3}u_{k,n}(y)=\frac{1}{\delta_{k}}\nabla_{3}u_{n}(x_{0}+\delta_{k}y_{\alpha},y_{3}),\quad [u_{k,n}] (y)  =\frac{1}{\delta_{k}} [u_{n}](x_0+\delta_ky_\alpha,y_3),
\end{equation*}
and so, recalling $(H_3)$,
\begin{equation*}
\frac{\de\mu}{\de\mathcal{L}^{2}}(x_{0})=\lim_{k,n}\bigg\{\int_{Q}W_{3d}\Big(\nabla_{\alpha}u_{k,n}\Big|\frac{\delta_{k}\nabla_{3}u_{k,n}}{\eps_{n}}\Big) \de y+\int_{Q\cap S(u_{k,n})} h_{3d}([u_{k,n}],\tilde\nu_{\alpha}(u_{n,k}))\de\mathcal{H}^{2}(y)\bigg\}.
\end{equation*}

Choose $n(k) \in\mathbb{N}$ such that ${\eps}_{k}':=\delta_k^{-1}\eps_{n(k)}\rightarrow0$; we have that the sequence $v_{k}(\cdot):=u_{k,n(k)}(\cdot)$ converges in $L^{1}$  to $\nabla_{\alpha}\overline{u}(x_0)(\cdot)$ by \eqref{827} and, by \eqref{830},
\begin{equation}\label{777}
\int_{I}\frac{\nabla_{3}v_{k}(y)}{\eps_{k}'}\,\de y_{3}\rightharpoonup\overline{d}(x_{0}) \qquad\text{in $L^p(\omega;\R3)$}.
\end{equation}
Then
\begin{equation*}
\frac{\de\mu}{\de\mathcal{L}^{2}}(x_{0}) =\lim_{k\to\infty}\bigg\{\int_{Q}W_{3d}\Big(\nabla_{\alpha}v_{k}\Big|\frac{\nabla_{3}v_{k}}{\eps_{k}'}\Big)\de y +\int_{Q\cap S(v_{k})}h_{3d}([v_{k}],\tilde\nu(v_{k}))\,\de \mathcal{H}^{2}(y)\bigg\}.
\end{equation*}
Next, we change slightly the sequence, in order to comply with the boundary condition in \eqref{107b}. We follow similar arguments to what is done in \cite{CF}.
Let $Q_{j}':=\{y_\alpha\in Q^{\prime}:\operatorname*{dist}(
y_\alpha,\partial Q^{\prime})>\frac{1}{j}\}  $ such that
$$\lim_{k\to\infty}\int_{\partial (Q_{j}'\times I)}\left\vert \nabla\overline{u}(x_{0})  y_{\alpha}-v_{k}(y_{\alpha},y_{3})\right\vert\, \de\mathcal{H}^{2}(y)=0$$
and define
$$
v_{k,j}(y):=\begin{cases}
v_{k}(y)   &  \text{in }Q_{j}'\times I,\\
\nabla_\alpha\overline{u}(x_{0})  y_{\alpha}  & \text{in }(Q^{\prime}\setminus Q_{j}')\times I.
\end{cases}$$
Clearly, $v_{k,j}\rightarrow v_{k}$ in $L^{1}(Q;\R3)$ as $j\to\infty$, and therefore, recalling $(H_1)$ and $(H_2)$,
$$\frac{\de\mu}{\de\mathcal{L}^{2}}(x_{0})  \geq\lim_{k,j}\bigg\{\int_{Q}W_{3d}\Big(\nabla_{\alpha}v_{k,j}\Big|\frac{\nabla_{3}v_{k,j}}{\eps_{k}'}\Big) \de y+\int_{Q\cap S(v_{k,j})}h_{3d}([v_{k,j}],\tilde\nu_{\alpha}(v_{k,j})) \,\de\mathcal{H}^{2}(y)  \bigg\}.$$
Following our argument in Step 1, for fixed $k$ we apply Theorem \ref{Al} to construct a function $g_{k,j}\in SBV(Q; \R3)$ such that 
$\nabla g_{k,j}=\eps_k'\Big(0\Big|\overline{d}(x_{0})-\int_I \frac{\nabla_{3}v_{k,j}}{\eps_{k}'} \,\de y_3\Big)$ 
and $\lVert g_{k,j}\rVert_{L^1(Q;\R3)}\leq C \eps_k'\left\lVert\overline{d}(x_{0})-\int_I \frac{\nabla_{3}v_{k,j}}{\eps_k'} \,\de y_3\right\rVert_{L^1(Q';\R3)}$. 
It is not difficult to verify that the function $w_{k,j}:=v_{k,j}+g_{k,j}$ is a competitor for $W_{3d,2d}(\nabla_{\alpha}\overline{u}(x_{0}) |\overline{d}(x_{0}))$, so that, recalling again $(H_1)$ and $(H_2)$,
\begin{equation*}
\begin{split}
\frac{\de\mu}{\de\mathcal{L}^{2}}(x_{0})& \geq\lim_{k,j}\bigg\{\int_{Q}W_{3d}\Big(\nabla_{\alpha}w_{k,j}\Big|\frac{\nabla_{3}w_{k,j}}{\eps_{k}'}\Big)\de y+\int_{Q\cap S(w_{k,j})}h_{3d}([w_{k,j}],\tilde\nu_{\alpha}(w_{k,j}))\de\mathcal{H}^{2}(y)\Big\} \\
& \geq W_{3d,2d}(\nabla_{\alpha}\overline{u}(x_{0})|\overline{d}(x_{0})),
\end{split}
\end{equation*}
which proves \eqref{778}.

\paragraph{\textit{\textbf{Step 4} (Lower bound -- surface)}}
Consider the sequence of functions $u_n\in SBV(\Omega;\R3)$ as at the beginning of Step 3, and let $\mu_n$ be the corresponding sequence of Radon measures given by $\eqref{miu_n}$.
Recalling that $\mu$ is their weak-* limit, we claim
that for $\mathcal{H}^{1}\res S(\cl u)$-a.e. $x_{0}\in S(\cl u)$
\begin{equation}\label{779}
\frac{\de{\mu}}{\de(\vert [\cl u] \vert\mathcal{H}^{1}\res S(\cl u))}(x_0)\geq\frac{1}{\vert [\cl u] \vert(x_0)}h_{3d,2d}([\cl u](x_0),\nu(\cl u)(x_0)).
\end{equation}
Since $\left( \nabla_\alpha u_n\Big| \frac{\nabla_3 u_n}{\eps_n}\right)$ is bounded in $L^p(\Omega;\R{3\times3})$, we have that $\nabla u_n \wto (H|0)$ in $L^p(\Omega;\R{3\times3})$ (up to a subsequence), for some $H \in L^p(\omega; \R{3{\times}2})$.
Let $x_0\in\omega\cap S(\cl u)$ be such that $\displaystyle\frac{\de\mu}{\de\mathcal{H}^{1}\res S(\cl u)}(x_0)$ exists, and consider a sequence $\delta_{k}\to0$ such that, denoting $\nu:=\nu(\cl u)(x_0)$,
\begin{gather*}
\lim_{k\to\infty}\vert[\overline{u}]\vert\mathcal{H}^{1}(S(\overline{u})\cap Q'_{\nu}(x_{0},\delta_{k}))=\vert[\overline u]\vert(x_0),\\
\lim_{k\to\infty}\frac{1}{\delta_{k}}\int_{Q'_{\nu}(x_0,\delta_{k})}\left\vert H(x_{\alpha})  \right\vert
\de x_{\alpha}=0.
\end{gather*}
Then
\begin{equation*}
\begin{split}
&\frac{\de\mu}{\de(\vert [\cl u] \vert\mathcal{H}^{1}\res S(\cl u))}(x_0) =\frac{1}{|[\overline{u}]|(x_0)}\lim_{k,n}\frac{1}{\delta_{k}}\bigg\{\int_{Q_\nu^{\prime}(x_0,\delta_k)\times I}W_{3d}\Big(\nabla_{\alpha}u_{n}\Big|\frac{\nabla_{3}u_{n}}{\eps_{n}}\Big)\de x  \\
& \qquad \phantom{\frac{1}{|[\overline{u}]|(x_0)}\lim_{k,n}\frac{1}{\delta_{k}}\quad} +\int_{(Q_\nu^{\prime}(x_{0},\delta_{k})\times I)\cap S(u_n) }h_{3d}([u_{n}],\tilde\nu_{\alpha}(u_{n})) \de\mathcal{H}^{1}(x_\alpha)\de x_3\bigg\} \\
&\quad=\frac{1}{\vert [u]\vert(x_{0})}\lim_{k,n}\bigg\{\delta_{k}\int_{Q_{\nu}^{\prime}\times I} W_{3d}\Big(\nabla_{\alpha}u_{n}(x_{0}+\delta_{k}y_{\alpha},y_{3})\Big|\frac{\nabla_{3}u_{n}(x_{0}+\delta_{k}y_{\alpha},y_{3})}{\eps_{n}}\Big)\de y\\
&\qquad+ \int_{(Q_{\nu}^{\prime}\times I)\cap \{y_{\alpha}:(x_{0}+\delta_{k}y_{\alpha},y_{3})  \in S(u_n)\}}h_{3d}([u_{n}](x_{0}+\delta_{k}y_{\alpha},y_{3}),\tilde\nu(u_{n})(x_{0}+\delta_{k}y_{\alpha},y_{3}))\de \mathcal{H}^{1}(y_\alpha)\de y_3\bigg\} \\
&\quad= \frac{1}{|[u]|(x_0)}\lim_{k,n}\bigg\{\int_{Q_{\nu}^{\prime}\times I}W_{3d}\Big(\frac{\nabla_{\alpha}u_{k,n}}{\delta_k}\Big| \frac{\nabla_{3}u_{k,n}}{\eps_{n}}\Big)\de y \\
&\quad \phantom{\frac{1}{|[u]|(x_0)}\lim_{k,n}\quad} + \int_{(Q_{\nu}^{\prime}\times I)\cap S(u_{k,n})}h_{3d}([u_{k,n}],\tilde\nu(u_{k,n})) \de\mathcal{H}^{1}(y_\alpha)\de y_3\bigg\},
\end{split}
\end{equation*}
where $u_{n, k}(y) := u_n(x_0 + \delta_ky_\alpha, y_3) - (\cl{u})^-(x_0)$.
By a diagonalization argument let $v_{k}:=u_{k,n(k)}$ so that $\lim_{k,n}\left\Vert v_k-\gamma_{[\overline{u}](x_0),\nu}\right\Vert _{L^{1}(Q'_\nu \times I)}=0$, $\nabla v_k \wto 0$ in $L^p(Q'_\nu \times I; \R3)$ and
 
 $$\frac{\de\mu}{\de(\vert [\cl u] \vert\mathcal{H}^{1}\res S(\cl u))}(x_0) \geq \frac{1}{|[\overline{u}]|(x_0)}\liminf_{k \to \infty} \int_{(Q_{\nu}^{\prime}\times I)\cap S(v_k)} h_{3d}( [v_k], \tilde\nu_{\alpha}(v_k))\, \de \mathcal{H}^2(y).$$
Following the arguments in \cite[Proposition 4.2]{CF}, we can obtain a new sequence $w_k$ which is a competitor for the cell problem \eqref{107c}, which implies \eqref{779}.
This concludes the proof of Theorem \ref{T3d2d}.
\qed

\subsection{Structured deformations}
In oder to pass to structured deformation for the functional in \eqref{F3d2d}, we shall use the relaxation theory developed in \cite{CF} to obtain the representation Theorem \ref{secondleft}.
Given $(\cl g,\cl G,\cl d)\in SBV(\omega;\R3){\times}L^1(\omega;\R{3{\times}2})\times L^p(\omega; \R3)$, we define the relaxed energy
\begin{equation}\label{200}
\begin{split}
\cF_{3d,2d,SD}(\cl g,\cl G,\cl d) :=  \inf \bigg\{ & \liminf_{n\to\infty} \bigg(\int_{\omega } W_{3d,2d}(\nabla u_n, \cl d)\,\de x_\alpha + \int_{\omega\cap S ( u_{n})} h_{3d,2d}([ u_n],\nu(u_n))\de \cH^1(x_\alpha)\bigg): \\
&  u_n \in SBV(\omega;\R3),\;  u_n\to \cl g \text{ in } L^1(\omega;\R3),\; \nabla u_n\wto \cl G \text{ in } L^p(\omega;\R{3\times2}) \bigg\}.
\end{split}
\end{equation}
\begin{remark}\label{pwc}
We notice that the presence of the field $\cl d$ in \eqref{F3d2d} introduces a dependence $x\mapsto W_{3d,2d}(A,\cl d(x))$ of the bulk density on the space variable $x$ not covered in \cite{CF}.
One approach to incorporate such a dependence on $x$ is to require that $x\mapsto W_{3d,2d}(A,\cl d(x))$ be continuous.
Such a continuity requirement was introduced in \cite{BBBF}.
To apply directly the results contained in \cite{BBBF}, we would need to impose a stronger regularity on the field $\cl d$, namely, we would have to require $\cl d\in C(\omega;\R3)$.
We avoid this by applying the technique presented in \cite{MS}: we approximate $\cl d$ by a sequence of piecewise constant functions $\cl d_k\in L^p(\omega;\R3)$, and we exploit the property \eqref{wishful} of the bulk energy density $W_{3d,2d}$ and the approximation result provided in \cite[Lemma 2.9]{CF}.
\end{remark}

Without writing the details of the proof, we assert that these observations, together with Proposition \ref{p100}, allow us to establish the following representation theorem.
\begin{theorem}\label{secondleft}
Under the hypotheses $(H_1)$--$(H_4)$, for each $(\cl g,\cl G,\cl d)\in SBV(\omega;\R3){\times}L^1(\omega;\R{3{\times}2})\times L^p(\omega; \R3)$,
the energy $\cF_{3d,2d,SD}(\cl g,\cl G,\cl d)$ admits an integral representation of the form:
\begin{equation}\label{201}
\cF_{3d,2d,SD}(\cl g,\cl G,\cl d)=\int_\omega W_{3d,2d,SD}(\nabla\cl g,\cl G,\cl d)\,\de x_\alpha+ \int_{\omega\cap S(\cl g)} h_{3d,2d,SD} ([\cl g],\nu (\cl g))\,\de\cH^1(x_\alpha),
\end{equation}
where, for $A , B \in \R{3{\times}2}, d \in \R3,$
\begin{equation}\label{2011}
\begin{split}
W_{3d,2d,SD}(A, B, d) := \inf \bigg\{ &\int_{Q'} W_{3d,2d} (\nabla u(x_\alpha), d)\, \de x_\alpha + \int_{Q'\cap S(u)} h_{3d,2d}([u], \nu(u))\, \de \cH^1(x_\alpha):\\
& u \in SBV(Q'; \R3),\; u|_{\partial Q'} = Ax_\alpha, \; \int_{Q'} \nabla u \,\de x_\alpha = B, \; |\nabla u| \in L^p(Q') \bigg\}
\end{split}
\end{equation}
and, for $\lambda \in \R3$ and $\eta \in \S1,$
\begin{equation}\label{2012}
\! h_{3d, 2d, SD} (\lambda,\eta)\!:=\inf \bigg\{ 
\!\int_{Q'_{\eta}\cap S(u)} \!\! h_{3d, 2d} ([u], \nu(u))\, \de\cH^1(x_\alpha): 
u \in SBV(Q'_{\eta}; \R3), \nabla u = 0, u|_{\partial Q'_{\eta}}= \gamma_{\lambda, \eta} \bigg\}.
\end{equation}
\end{theorem}

\section{The right-hand path}\label{sect:RHS}
In this section we relax our initial energy \eqref{E3d} by first passing to structured deformations and then carrying out the dimension reduction.

\subsection{Structured deformations}
For $g \in SBV(\Omega_\eps; \R3)$ and $G^{\backslash 3} \in L^1(\Omega_\eps; \R{3{\times}2})$, define
\begin{equation}\label{2018}
\begin{split}
 \mathcal{F}_{3d,SD}( g,G^{\backslash 3}) := \inf \bigg\{ & \liminf_{n\to\infty} \bigg(\int_{\Omega_\eps} W_{3d}(\nabla u_n) \,\de x + \int_{\Omega_\eps\cap S(u_n)} h_{3d} ( [u_n], \nu(u_n))\, \de\mathcal{H}^2(x)\bigg): \\
& u_n \to g \text{ in $L^1(\Omega_\eps;\R3)$}, \nabla u_n \wto (G^{\backslash 3} | \nabla_3 g) \,\, \text{in}\, \, L^p(\Omega_\eps;\R{3\times3}) \bigg\}.
\end{split}
\end{equation}
An integral representation for $ \mathcal{F}_{3d, SD}$ follows immediately from \cite[Theorem 2.17]{CF}.
As stated in Remark \ref{energies}, the coercivity assumption \eqref{coerc} grants boundedness of the gradients $\nabla u_n$ in $L^p$, thereby justifying the choice of weak convergence of $\{\nabla u_n\}$ in the definition \eqref{2018}.
In that definition, we are considering the case in which the limit is classical in the third component of the gradient, that is $\nabla_3 u_n\wto \nabla_3 g$.

\begin{theorem}\label{firstSD}
Under the hypotheses $(H_1)$--$(H_4)$, for $g \in SBV(\Omega_\eps; \R3)$ and $G^{\backslash 3} \in L^1(\Omega_\eps; \R{3{\times}2})$, the functional $\mathcal{F}_{3d,SD}( g,G^{\backslash 3})$ admits an integral representation of the form:
\begin{equation}\label{2017}
\mathcal{F}_{3d,SD}(g,G^{\backslash 3}) = \int_{\Omega_\eps} W_{3d,SD} (\nabla g, G^{\backslash 3}) \,\de x + \int_{\Omega_\eps\cap S(g)} h_{3d,SD}([g],\nu(g))\,\de\mathcal{H}^2(x),
\end{equation}
where, for $A \in \R{3{\times} 3}$ and $B^{\backslash 3} \in \R{3{\times }2}$,
\begin{equation}\label{2015}
\begin{split}
W_{3d,SD} (A, B^{\backslash 3} ) = \inf \bigg\{ & \int_Q W_{3d} (\nabla u)\,\de x + \int_{Q \cap S(u)} h_{3d} ([u],\nu(u))\,\de\mathcal {H}^2(x): \\
& u \in SBV(Q;\R3), \; u| _{\partial Q}= Ax, \; |\nabla u | \in L^p(Q), \; \int_Q \nabla u\,\de x = (B^{\backslash 3} | Ae_3)\bigg\}
\end{split}
\end{equation}
and, for $\lambda \in \R3$, $\nu \in \S2,$
\begin{equation}\label{2016}
h_{3d,SD} (\lambda, \nu)=\inf\bigg\{\int_{Q_\nu}h_{3d}([u], \nu(u))\,\de\mathcal{H}^2(x): \; u \in SBV(Q_\nu ; \R3), \nabla u = 0 \text { a.e.}, \; u|_{ \partial Q_\nu} = \gamma_{\lambda, \nu}\bigg\}.
\end{equation}
\end{theorem}

\begin{proposition}\label{energiesCF}
Let $W_{3d,SD}$ and $h_{3d,SD}$ be defined by \eqref{2015} and \eqref{2016}, respectively. Then
\begin{itemize}
\item[(i)] $W_{3d,SD}$ is locally Lipschitz continuous separately in $A$ and $B^{\backslash3}$, namely for every $B^{\backslash3}\in\R{3\times2}$ and every $A_1\in\R{3\times3}$ there exists a constant $C_1>0$ such that
$$|W_{3d,SD}(A_1,B^{\backslash 3})-W_{3d,SD}(A_2,B^{\backslash 3})|\leq C_1|A_1-A_2|$$
whenever $|A_1-A_2|$ is small enough; in particular, 
$$|W_{3d,SD}(A_1,B^{\backslash 3})-W_{3d,SD}(A_2,B^{\backslash 3})|\leq C_1|A_1-A_2|(1+|A_1|^{p-1}+|A_2|^{p-1}).$$

\noindent Similarly, for every $A\in\R{3\times3}$ and $B_1^{\backslash3}\in\R{3\times2}$ there exists a constant $C_2>0$ such that
$$|W_{3d,SD}(A,B_1^{\backslash 3})-W_{3d,SD}(A,B_2^{\backslash 3})|\leq C_2|B_1^{\backslash 3}-B_2^{\backslash 3}|$$
whenever $|B_1^{\backslash 3}-B_2^{\backslash 3}|$ is small enough; in particular, 
$$|W_{3d,SD}(A,B_1^{\backslash 3})-W_{3d,SD}(A,B_2^{\backslash 3})|\leq C_2|B_1^{\backslash 3}-B_2^{\backslash 3}|(1+|B_1^{\backslash 3}|^{p-1}+|B_2^{\backslash 3}|^{p-1});$$
\item[(ii)] $h_{3d,SD}$ satisfies $(H_2)$--$(H_4)$.
\end{itemize}
\end{proposition}
\begin{proof}
The proof of part (i) follows that of \cite[Proposition 5.2]{CF}; part (ii) follows from the corresponding properties of $h_{3d}$. \qed
\end{proof}

\subsection{Dimension reduction}
We now apply dimension reduction to the energy $\cF_{3d,SD}$ defined in \eqref{2017}.
As we did in Section \ref{LHSDR}, we rescale the variables by $(x_\alpha,x_3)\mapsto(x_\alpha,x_3/\eps)$, thereby replacing the domain of integration $\Omega_\eps$ by $\Omega$, and we rescale the energy $\cF_{3d,SD}$ by dividing it by $\eps$.
Therefore, given $(\overline{g}, \overline{G}, \overline{d}) \in SBV (\omega; \R3) \times L^1(\omega; \R{3{\times}2}) \times L^p(\omega; \R3)$, we seek an integral representation for the following relaxed energy
\begin{equation}\label{2019}
\begin{split} 
\mathcal{F}_{3d, SD, 2d}( \overline{g}, \overline{G}, \overline{d}) := \inf \bigg\{ & \liminf_{n\to\infty} \bigg(\int_\Omega W_{3d,SD} \Big(\Big(\nabla_\alpha g_n \Big| \frac{\nabla_3 g_n}{\eps_n}\Big), \cl G\Big)\de x \\
&\phantom{\liminf}+ \int_{\Omega\cap S(u_n)} h_{3d, SD} \Big([g_n], \Big(\nu_\alpha (g_n)\Big| \frac{\nu_3(g_n)}{\eps_n}\Big)\Big)\de \mathcal{H}^2(x)\bigg): \\
& g_n\to \overline{g} \text{ in } L^1(\Omega;\R3),\;  \int_I \frac{\nabla_3 g_n}{\eps_n} \,\de x_3 \wto \overline{d}\text{ in } L^p(\omega;\R3),\; \nu(g_n)\cdot e_3 = 0 \bigg \}.
\end{split}
\end{equation}
An analogue of Remark \ref{pwc} can be made with the roles of $\cl G$ and $\cl d$ interchanged and with Proposition \ref{energiesCF} in place of Proposition \ref{p100}, and this provides a proof of the following representation theorem.

\begin{theorem}\label{secondright}
Under the hypotheses $(H_1)$--$(H_4)$, given $(\overline{g}, \overline{G}, \overline{d}) \in SBV (\omega; \R3) \times L^1(\omega; \R{3{\times}2}) \times L^p(\omega; \R3)$,
the relaxed energy $\mathcal{F}_{3d, SD, 2d}$ defined in \eqref{2019} admits the integral representation
\begin{equation}\label{301}
\mathcal{F}_{3d, SD, 2d}(\cl g, \cl G, \cl d) = \int_\omega W_{3d, SD, 2d} (\nabla \cl g, \cl G, \cl d) \; \de x_\alpha + \int_{\omega\cap S(\cl g)} h_{3d, SD, 2d} ([\cl g], \nu(\cl g))\; \de \cH^1(x_\alpha),
\end{equation}
where, for $A, B \in \R{3{\times}2}, d \in \R3,$
\begin{equation}\label{2013}
\begin{split}
\!\! W_{3d, SD, 2d}(A, B, d):= \inf \bigg\{ & \int_{Q'} W_{3d, SD} ((\nabla u(x_\alpha)|z(x_\alpha)),B) \,\de x_\alpha + \int_{Q' \cap S(u)} \!\! h_{3d, SD} ([u], \tilde\nu(u))\; \de \cH^1(x_\alpha): \\
& u \in SBV(Q'; \R3),\; |\nabla u| \in L^p(Q'), \; u|_{\partial Q'} = Ax_\alpha, \\
& z \in L^p_{Q' - \mathrm{per}}(\R2; \R3),\; \int_{Q'} z \,\de x_\alpha = d \bigg\},
\end{split}
\end{equation}
and, for $\lambda \in \R3, \eta \in \S1$,
\begin{equation}\label{2014}
h_{3d, SD, 2d}(\lambda, \eta)= \inf \bigg\{ \int_{Q'_\eta} h_{3d, SD}( [u], \nu(u))\; \de \cH^1(x_\alpha): \; u \in SBV(Q'; \R3), \; \nabla u = 0\; a.e., \; u|_{\partial Q'} = \gamma_{\lambda, \eta}\bigg\}.
\end{equation}
\end{theorem}

\section {Comparison of the relaxed energy densities for the left- and right-hand paths}\label{sect:comp_ex}
In this section we discuss the relationship between the doubly relaxed energy densities \eqref{201} and \eqref{301} obtained in Sections \ref{sect:LHS} and \ref{sect:RHS}.
At present, at the level of generality of Theorems \ref{secondleft} and \ref{secondright}, an explicit comparison in terms of whether one of the two energies is smaller than the other is not available.
Nonetheless, quantitative results can be obtained when the initial energy \eqref{E3d} has a a specific form, namely it is a purely interfacial energy ($W_{3d}=0$) with a specific choice of the interfacial energy density $h_{3d}$.

Our aim then is to compute explicitly the densities provided by the cell formulas \eqref{107b}, \eqref{107c}, \eqref{2011}, \eqref{2012}, \eqref{2015}, \eqref{2016}, \eqref{2013}, and \eqref{2014} starting from the initial, purely interfacial, energy density (see \cite{nogap,OP15})
\begin{equation}\label{purelyinterfacial}
h_{3d}(\lambda,\nu)=|\lambda \cdot \nu|.
\end{equation}

\paragraph{\textbf{The left-hand path}}
Let us consider \eqref{purelyinterfacial} and let $(A,d)\in\R{3\times2}\times\R3$; then \eqref{107b} reads
\begin{equation}\label{1231}
W_{3d,2d}(A,d)= \inf\bigg\{\int_{Q^{\prime}\cap S(u)}|[u]\cdot\tilde\nu(u)|\,\de\mathcal{H}^{1}(x_\alpha): u\in SBV(Q^{\prime};\mathbb{R}^{3}), u|_{\partial Q^{\prime}}(x_{\alpha})=Ax_{\alpha}\bigg\}	=0.
\end{equation} 
The first equality is a consequence of \eqref{purelyinterfacial}; the second one follows since the affine function $u(x_\alpha) = Ax_\alpha$ is admissible and makes the integral vanish.

Let us now turn to \eqref{107c}: we claim that for $\lambda\in\R3$, $\eta\in\mathbb{S}^1$, the surface energy density $h_{3d,2s}$ reads
\begin{equation}\label{107c1}
h_{3d,2d}(\lambda,\eta)=|\lambda\cdot\tilde\eta|.
\end{equation}
In fact, the function $u(x_\alpha)=\gamma_{\lambda,\eta}(x_\alpha)$ (see \eqref{gammale}) is admissible and it provides an upper bound; to obtain a lower bound, one uses the following version of the Gauss-Green formula in $SBV$ (see \cite[Theorem 3.36]{AFP} and also \cite{DPO95,V,VH}): for $u\in SBV(\Omega;\R3)$ and $U\subset\Omega$, there holds 
\begin{equation}\label{GGformula}
\int_{U\cap S(u)} [u]\cdot\nu(u)\,\de\cH^{2}(x)+\int_U \div u\,\de x-\int_{\partial U} u\cdot\nu_U\,\de\cH^2(x)=0.
\end{equation}
Considering the integrand in \eqref{107c}, by using the properties of the absolute value and \eqref{GGformula}, the same $u(x_\alpha)=\gamma_{\lambda,\eta}(x_\alpha)$ gives
\begin{equation}\label{107c2}
\int_{Q_\eta'\cap S(u)} |[u]\cdot\tilde\nu(u)|\,\de\cH^1(x_\alpha)\geq \left|\int_{Q_\eta'\cap S(u)} [u]\cdot\tilde\nu(u)\,\de\cH^1(x_\alpha)\right|= \left|\int_{\partial Q_\eta'} u\cdot\tilde\nu_{Q_\eta'}\,\de\cH^1(x_\alpha)\right| = |\lambda\cdot\tilde\eta|,
\end{equation}
which completes the proof of \eqref{107c1}.
Given $(\cl{u},\cl{d}) \in SBV(\omega;\mathbb{R}^{3}) \times L^{p}(\omega ;\mathbb{R}^{3})$, the relaxed energy \eqref{F3d2d} reads then
\begin{equation}\label{ex3d2d}
\cF_{3d,2d}(\cl{u},\cl{d})=\widehat\cF_{3d,2d}(\cl u):=\int_{\omega \cap S(\cl{u})} |[\cl{u}]\cdot\tilde\nu(\cl{u})|\,\de\mathcal{H}^{1}(x_\alpha),
\end{equation}
where we notice that the dependence on $\cl d$ is lost.

Next, we claim that, for $A,B\in\R{3\times2}$, $d\in\R3$, the bulk density \eqref{2011} is given by $W_{3d,2d,SD}(A, B, d)=\widehat W_{3d,2d,SD}(A,B)$, which is the relaxation of $h_{3d,2d}$ in \eqref{107c1}, and reads
\begin{equation}\label{2011a}
\begin{split}
\widehat W_{3d,2d,SD}(A, B)=& \inf \bigg\{ \int_{Q'\cap S(u)} |[u]\cdot \tilde\nu(u)|\, \de \cH^1(x_\alpha):  u \in SBV(Q'; \R3), \; u|_{\partial Q'} = Ax_\alpha, \\
& \phantom{\inf\Bigg\{} \int_{Q'} \nabla u \,\de x_\alpha = B, \; |\nabla u| \in L^p(Q') \bigg\};
\end{split}
\end{equation}
notice again that this is independent of $d$.
We prove that, for $A,B\in\R{3\times2}$,
\begin{equation}\label{2011b}
\widehat W_{3d,2d,SD}(A, B)=\big|\tr\big((A|0)-(B|0)\big)\big|=|A_{11}+A_{22}-B_{11}-B_{22}|.
\end{equation}
Again as before, we prove \eqref{2011b} by obtaining upper and lower bounds for $\widehat W_{3d,2d,SD}$.
Let $u$ be an admissible function for \eqref{2011a} and define $u_\alpha:Q' \to \R2$ by $u_\alpha(x_\alpha):= (u_1(x_\alpha), u_2(x_\alpha))$.
Since
\begin{equation}\label{1292}
\int_{Q' \cap S(u)} |[u]\cdot\tilde{\nu}(u)|\, \de \cH^1(x_\alpha) = \int_{Q' \cap S(u_\alpha)} |[u_\alpha]\cdot\nu(u_\alpha)|\, \de \cH^1(x_\alpha),
\end{equation}
the function $u_\alpha$ is admissible for the minimum problem
\begin{equation}\label{1293}
 \inf \bigg\{ \int_{Q'\cap S(v)} \!\! |[v]\cdot \nu(v)|\, \de \cH^1(x_\alpha):  v\in SBV(Q'; \R2), \, v|_{\partial Q'} = \widehat{A}x_\alpha,\, \int_{Q'} \nabla v\,\de x_\alpha = \widehat{B},\, |\nabla v| \in L^p(Q') \bigg\}, 
\end{equation}
where $\widehat{A}$ and $\widehat{B}$ denote the upper $2{\times}2$ sub-matrices of $A$ and $B$, respectively.
The lower bound for $\widehat W_{3d,2d,SD}$ then follows immediately from the result in \cite{nogap,OP15}, where it is proved that the infimum in \eqref{1293} is given by $|\tr(\widehat A-\widehat B)|$.

In order to derive the upper bound for $\widehat W_{3d,2d,SD}$, fix $\eps>0$ and let $v_\epsilon\in SBV(Q';\R2)$ admissible for \eqref{1293} be such that
\begin{equation}\label{1311}
\int_{Q'\cap S(v_\epsilon)} |[v_\epsilon]\cdot \nu(v_\epsilon)|\, \de \cH^1(x_\alpha) \leq |\tr (\widehat{A} - \widehat{B})| + \epsilon.
\end{equation}
Using Lemma $4.3$ in \cite{M07}, we can construct a function $v \in SBV(Q')$ such that 
$$ v|_{\partial Q'} = e_3\cdot Ax_\alpha, \qquad \nabla v = (B_{31}, B_{32})\quad \text{$\cL^2$-a.e. in $Q'$}.$$
Then, the function $ w_\epsilon\in SBV( Q' ; \R3)$ defined by $w_\epsilon(x_\alpha) := (v_\epsilon(x_\alpha), v(x_\alpha))$ is admissible for \eqref{2011a}, and by \eqref{1292} and \eqref{1311} we conclude that
\begin{equation}\label{1319}
\int_{Q'\cap S(w_\epsilon)} |[w_\epsilon]\cdot \tilde{\nu}(w_\epsilon))|\, \de \cH^1(x_\alpha) \leq |\tr(\widehat{A} - \widehat{B})| + \epsilon,
\end{equation}
and the result follows from the arbitrariness of $\eps$.
Formula \eqref{2011b} is therefore proved.

Finally, we observe that the same strategy used to prove \eqref{107c1} can be used to show that for $\lambda \in \R3$ and $\eta \in \S1,$
\begin{equation}\label{2012a}
h_{3d, 2d, SD} (\lambda, \eta)= |\lambda\cdot\tilde\eta|.
\end{equation}
Thus, in view of \eqref{2011b} and \eqref{2012a}, given $(\cl g,\cl G,\cl d)\in SBV(\omega;\R3){\times}L^1(\omega;\R{3{\times}2})\times L^1(\omega; \R3)$, the functional $\cF_{3d,2d,SD}$ in \eqref{201} can be written as
\begin{equation}\label{ex3d2dSD}
\begin{split}
\cF_{3d,2d,SD}(\cl g,\cl G,\cl d)= & \widehat\cF_{3d,2d,SD}(\cl g,\cl G):= \int_\omega |\tr((\nabla\cl g|0)-(\cl G|0))|\,\de x_\alpha+\int_{\omega\cap S(\cl g)} |[\cl g]\cdot\tilde\nu(\cl g)|\,\de\cH^1(x_\alpha) \\
=& \int_\omega \Big|\frac{\partial\cl g_1}{\partial x_1}+\frac{\partial\cl g_2}{\partial x_2}-\cl G_{11}-\cl G_{22}\Big|\,\de x_\alpha+\int_{\omega\cap S(\cl g)} |[\cl g_1]\nu_1(\cl g)+[\cl g_2]\nu_2(\cl g)|\,\de\cH^1(x_\alpha).
\end{split}
\end{equation}

\paragraph{\textbf{The right-hand path}}
Considering \eqref{purelyinterfacial}, the explicit formulas for the energy densities $W_{3d, SD}$ and $h_{3d, SD}$ in \eqref{2015} and \eqref{2016} were derived in \cite{nogap,OP15} (see also \cite{S17}); denoting by $M^i$, $i=1,2,3$ the columns of a matrix $M\in\R{3\times3}$, for $A\in\R{3\times3}$ and $B^{\backslash 3} \in \R{3\times 2}$ we have that 
\begin{equation}\label{1193}
W_{3d, SD}( A, B^{\backslash 3}) = |\tr(A-(B^{\backslash 3}|A ^3))|,
\end{equation}
and, for $\lambda \in \R3$ and $\nu \in \S2$,
\begin{equation}\label{1194}
h_{3d,SD} (\lambda, \nu) = |\lambda\cdot\nu|.
\end{equation}
Therefore, for $(g,G^{\backslash 3})\in SBV(\Omega; \R3)\times L^1(\Omega; \R{3{\times}2})$, plugging \eqref{1193} and \eqref{1194} in \eqref{2017} gives
\begin{equation}\label{2017a}
\mathcal{F}_{3d,SD}( g,G^{\backslash 3}) = \int_\Omega \Big|\frac{\partial g_1}{\partial x_1}+\frac{\partial g_2}{\partial x_2}-G_{11}^{\backslash 3}-G_{22}^{\backslash 3}\Big|\,\de x + \int_{\Omega\cap S(g)} |[g]\cdot\nu(g)|\,\de\mathcal{H}^2(x),
\end{equation}
Let us now turn to \eqref{2013}.
Let $A, B \in \R{3{\times}2}$, $d \in \R3$, and let $(u,z)$ be an admissible pair of functions for the minimization problem that defines $W_{3d,SD,2d}$; using \eqref{1193} and \eqref{1194}, and again the properties of the absolute value and the Gauss-Green formula \eqref{GGformula}, we can estimate
\begin{equation}\label{1195}
\begin{split}
\!\! \int_{Q'} |\tr((\nabla u | z) - (B | z))|\, \de x_\alpha &+ \int_{Q' \cap S(u)} |[u]\cdot \tilde\nu(u)|\,\de \cH^1(x_\alpha)\\
& \geq \bigg|\int_{Q'} \tr((\nabla u | z)-(B | z))\,\de x_\alpha\bigg| +\bigg|\int_{Q' \cap S(u)} [u]\cdot\tilde\nu(u)\,\de \cH^1(x_\alpha)\bigg|\\
& \geq \bigg|\int_{Q'} \tr((\nabla u | z)-(B | z))\,\de x_\alpha+\int_{Q' \cap S(u)} [u]\cdot\tilde\nu(u)\,\de \cH^1(x_\alpha)\bigg|\\
& =\bigg|\tr\bigg(\int_{Q'} \nabla (u_1,u_2)\,\de x_\alpha+ \int_{Q' \cap S(u)} [u]\otimes\tilde\nu(u)\,\de \cH^1(x_\alpha)\bigg)-B_{11}-B_{22}\bigg|\\
& =\bigg|\tr\bigg(\int_{Q'} (u_1, u_2) \otimes \nu_{\partial Q'}\,\de \cH^1(x_\alpha)\bigg)-B_{11}-B_{22}\bigg|\\
& =|A_{11}+A_{22}-B_{11}-B_{22}|
\end{split}
\end{equation}
where the last equality follows from the condition $u|_{\partial Q'}(x_\alpha) = Ax_\alpha$.
Since the affine function $u(x_\alpha) = A x_\alpha$ is admissible, the lower bound \eqref{1195} is attained, so that the density in \eqref{2013} reads
\begin{equation}\label{2013a}
W_{3d, SD, 2d}(A, B, d)=|A_{11}+A_{22}-B_{11}-B_{22}|=|\tr(\widehat A-\widehat B)|=:\widehat W_{3d, SD, 2d}(A, B).
\end{equation}
Finally, with the same reasoning as before, it is easy to see that the infimum in \eqref{2014} is attained at $u(x_\alpha)=\gamma_{\lambda, \eta}(x_\alpha)$, so that
\begin{equation}\label{2014a}
h_{3d, SD, 2d}(\lambda,\eta)=|\lambda\cdot\tilde\eta|.
\end{equation}
Thus, in view of \eqref{2013a} and \eqref{2014a}, given $(\cl g,\cl G,\cl d)\in SBV(\omega;\R3){\times}L^1(\omega;\R{3{\times}2})\times L^1(\omega; \R3)$, the functional $\cF_{3d,SD,2d}$ in \eqref{301} can be written as
\begin{equation}\label{ex3dSD2d}
\begin{split}
\cF_{3d,SD,2d}(\cl g,\cl G,\cl d)= & \widehat\cF_{3d,SD,2d}(\cl g,\cl G):= \int_\omega |\tr(\widehat{\nabla\cl g}-\widehat{\cl G})|\,\de x_\alpha+\int_{\omega\cap S(\cl g)} |[\cl g]\cdot\tilde\nu(\cl g)|\,\de\cH^1(x_\alpha) \\
=& \int_\omega \Big|\frac{\partial\cl g_1}{\partial x_1}+\frac{\partial\cl g_2}{\partial x_2}-\cl G_{11}-\cl G_{22}\Big|\,\de x_\alpha+\int_{\omega\cap S(\cl g)} |[\cl g_1]v_1(\cl g)+[\cl g_2]v_2(\cl g)|\,\de\cH^1(x_\alpha).
\end{split}
\end{equation}
Notice that we have proved that the bulk energy densities in \eqref{2011b} and \eqref{2013a} coincide, and the same holds true for the surface energy densities \eqref{2012a} and \eqref{2014a}.
Thus, we have proved the following result.
\begin{proposition}\label{S1}
Let $W_{3d}=0$ and $h_{3d}$ as in \eqref{purelyinterfacial}.
Then, the doubly relaxed energies \eqref{201} and \eqref{301} coincide and are both given by \eqref{ex3d2dSD} or \eqref{ex3dSD2d}.
\end{proposition}

\section{A one-step approach to dimension reduction in the context of structured deformations}\label{sect:6}
In this section, we recall an alternative procedure for dimension reduction in the context of structured deformations already available in the literature \cite{MS}.
The basic function spaces considered for this approach are the spaces \cite{CLT1,CLT2}
\begin{equation}
\begin{split}
SBV^2(\Omega;\R3):= & \{u\in SBV(\Omega;\R3): \nabla u\in SBV(\Omega;\R{3\times3})\}, \\
BV^2(\Omega;\R3):= & \{u\in BV(\Omega;\R3): \nabla u\in BV(\Omega;\R{3\times3})\}.
\end{split}
\end{equation}
For a function $v\in SBV^2(\Omega_\eps;\R3)$, the initial energy considered in \cite{MS} is of the form
\begin{equation}\label{enMS}
E^{MS}_\eps(v):=\int_{\Omega_\eps} W(\nabla v,\nabla^2 v)\,\de x+\int_{\Omega_\eps\cap S(v)} \Psi_1([v],\nu(v))\,\de\cH^2(x)+\int_{\Omega_\eps\cap S(\nabla v)} \Psi_2([\nabla v],\nu(\nabla v))\,\de\cH^2(x),
\end{equation}
where the bulk energy density $W\colon\R{3\times3}{\times}\R{3\times3\times3}\to[0,+\infty)$ is continuous, coercive, and has growth of order $p=1$, and the surface energy densities $\Psi_1\colon\R3{\times}\S{2}\to[0,+\infty)$ and $\Psi_2\colon\R{3\times3}{\times}\S2\to[0,+\infty)$ are continuous, coercive, have growth of order $1$ and are also subadditive and homogeneous of degree $1$ in the first vadiable; see the assumptions $(H_1)$--$(H_8)$ in \cite{MS} for the precise details.
We also refer the reader to \cite[Introduction and Remark 1.5]{MS} for a justification of the presence of the second-order gradient in the bulk density and of the energy density $\Psi_2$.

The main result obtained in \cite{MS} is an integral representation result for the relaxed functional
\begin{equation}\label{S3}
I(g,b,G):=\inf\Big\{\liminf_{n\to\infty} J_{\eps_n}(u_n): u_n\in SBV^2(\Omega;\R3), u_n\stackrel{L^1}\to g, \frac1{\eps_n}\nabla_3 u_n\stackrel{L^1}\to b, \nabla_\alpha u_n\stackrel{L^1}\to G\Big\},
\end{equation}
where $(g,b,G)\in BV^2(\omega;\R3){\times}BV(\omega;\R3){\times}BV(\omega;\R{3\times2})$, $\eps_n$ is a sequence tending to zero from above, and the functional $J_{\eps_n}$ is obtained by rescaling $E_{\eps_n}^{MS}$ in \eqref{enMS} by $\eps_n$ in the third variable and then dividing by $\eps_n$, analogously to the definition of $F_\eps$ from $E_\eps$ in \eqref{107}.
The field $b$ plays the role of the field $\cl d$ in the previous sections.
One important difference between \cite{MS} and the present work is that the vector field $b$ in \eqref{S3} already depends only on $x_\alpha$ because of the coercivity conditions alone (see again \cite[assumptions $(H_1)$--$(H_8)$ and Remark 1.5]{MS}), whereas in the previous sections it was necessary to average in the $x_3$ variable.
Moreover, it is evident that the process of relaxation in \eqref{S3} is a simultaneous passage to structured deformations and dimension reduction.
\begin{theorem}[{\cite[Theorem 1.4]{MS}}]\label{S4}
The functional $I$ defined in \eqref{S3} does not depend on the sequence $\{\eps_n\}$ and admits an integral representation of the form $I=I_1+I_2$, where, for $(g,G)\in BV^2(\omega;\R3){\times}BV(\omega;\R{3\times2})$,
\begin{equation}\label{S5}
I_1(g,G)=\int_\omega W_1(G-\nabla g)\,\de x_\alpha+\int_\omega W_1\bigg(-\frac{\de D^cg}{\de|D^cg|}\bigg)\,\de|D^cg|(x_\alpha)+\int_{\omega\cap S(g)} \Gamma_1([g],\nu(g))\,\de\cH^1(x_\alpha)
\end{equation}
and for $(b,G)\in BV(\omega;\R3){\times}BV(\omega;\R{3\times2})$
\begin{equation}\label{S6}
\begin{split}
I_2(b,G)= & \int_\omega W_2(b,G,\nabla b,\nabla G)\,\de x_\alpha+\int_\omega W_2^\infty\bigg(b,G,\frac{\de D^c(b,G)}{\de|D^c(b,G)|}\bigg)\,\de|D^c(b,G)| \\
& +\int_{\omega\cap S((b,G))} \Gamma_2((b,G)^+,(b,G)^-,\nu((b,G)))\,\de\cH^1(x_\alpha).
\end{split}
\end{equation}
The energy densities of $I_1$ are obtained as follows: for each $A\in\R{3\times2}$, $\lambda\in\R3$, and $\eta\in\S1$,
\begin{align}
W_1(A)&=\inf\bigg\{\int_{Q'\cap S(u)} \cl\Psi_1([u],\nu(u))\,\de\cH^1(x_\alpha): u\in SBV(Q';\R3), u|_{\partial Q'}=0, \nabla u=A\; a.e.\bigg\}, \label{W1} \\
\Gamma_1(\lambda,\eta)&=\inf\bigg\{\int_{Q_\eta'\cap S(u)} \cl\Psi_1([u],\nu(u))\,\de\cH^1(x_\alpha): u\in SBV(Q_\eta';\R3), u|_{\partial Q_\eta'}=\gamma_{\lambda,\eta}, \nabla u=0\; a.e.\bigg\}, \label{Gamma1} 
\end{align}
with $\gamma_{\lambda,\eta}$ defined as in \eqref{gammale} and
\begin{equation}\label{S9}
\cl\Psi_1(\lambda,\nu):=\inf\{\Psi_1(\lambda,(\nu|t)):t\in\R{}\}.
\end{equation}
The energy densities of $I_2$ are obtained as follows: for each $A\in\R{3\times2}$, $B_\beta\in\R{3\times3\times2}$, $\Lambda,\Theta\in\R{3\times3\times2}$, and $\eta\in\S1$,
\begin{equation}\label{S10}
\begin{split}
W_2(A,B_\beta)=\inf\bigg\{ & \int_{Q'}\cl W(A,\nabla u)\,\de x_\alpha+\int_{Q'\cap S(u)} \cl\Psi_2([u],\nu(u))\,\de\cH^1(x_\alpha): \\
& u\in SBV(Q';\R{3\times3}), u_{ik}|_{\partial Q'}=\sum_{j=1}^2 B_{ijk}x_j\bigg\},
\end{split}
\end{equation}
\begin{equation}\label{S11}
\begin{split}
\Gamma_2(\Lambda,\Theta,\eta)=\inf\bigg\{ & \int_{Q_\eta'} \cl W^\infty(u,\nabla u)\,\de x_\alpha+\int_{Q_\eta'\cap S(u)} \cl\Psi_2([u],\nu(u))\,\de\cH^1(x_\alpha): \\
& u\in SBV(Q'_\eta;\R{3\times3}), u|_{\partial Q_\eta'}=u_{\Lambda,\Theta,\eta}\bigg\},
\end{split}
\end{equation}
where
\begin{equation}\label{S12}
u_{\Lambda,\Theta,\eta}(x_\alpha):=
\begin{cases}
\Lambda & \text{if $0\leq x_\alpha\cdot\eta<1/2$,} \\
\Theta & \text{if $-1/2<x_\alpha\cdot\eta<0$,}
\end{cases}
\end{equation}
and with $\cl W$ and $\cl\Psi_2$ as follows: decomposing $B\in\R{3\times3\times3}$ into $(B_\beta,B_3)\in\R{3\times3\times2}{\times}\R{3\times3\times1}$ (i.e., $B_\beta$ denotes $B_{ijk}$ with $k=1,2$), define
\begin{equation}\label{S13}
\cl W(A,B_\beta):=\inf\{W(A,(B_\beta,B_3)): B_3\in\R{3\times3\times1}\},
\end{equation}
and for $\Lambda\in\R{3\times3}$ and $\eta\in\S1$, let
\begin{equation}\label{S14}
\cl\Psi_2(\Lambda,\eta):=\inf\{\Psi_2(\Lambda,(\eta|t)):t\in\R{}\}.
\end{equation}
\end{theorem}
In the statement of Theorem \ref{S4}, a superscript ``$\infty$'' denotes the recession function at infinity (see \cite[hypothesis ($H_3$) on page 461]{MS}), whereas the superscript ``$c$''  denotes the Cantor part.
We also point out that we maintained the notation from \cite{MS} for the convenience of the reader; in the notations of our previous sections, the  triple $(g,b,G)$ would be written $(\cl g,\cl G,\cl d)$.

\smallskip

\noindent
\textit{Sketch of the proof of Theorem \ref{S4}} 
By making use of Theorem \ref{Al}, the relaxed functional $I$ defined in \eqref{S3} can be additively decomposed into the functionals $I_1$ and $I_2$ defined in \eqref{S5} and \eqref{S6}, respectively, decoupling the effects of the surface energy density $\Psi_1$ from the bulk energy density $W$ and the surface energy density $\Psi_2$ (see \cite[Section 3.1]{MS}).
The result is obtained by proving upper and lower bounds for the Radon-Nikod\'ym derivative of the energies $I_1$ and $I_2$.
The technique is analogous to that presented in detail in the proof of Theorem \ref{T3d2d} in Section \ref{sect:LHS}.
The lower bounds aim at proving the $\liminf$ inequality \eqref{BGLI}; the upper bounds aim at proving the $\limsup$ inequality \eqref{BGLS}.
\qed

\smallskip

We are not undertaking a comparison of the relaxed energy in Theorem \ref{S4} with those obtained in Theorems \ref{secondleft} and \ref{secondright} at this level of generality, however, we do so for the particular choice made in Section \ref{sect:comp_ex}, namely for an initial energy where the only non-zero contribution comes from the jumps of the $SBV$ function, and not of its gradient, i.e., $W=\Psi_2=0$ and $\Psi_1(\lambda,\nu)=|\lambda\cdot\nu|$, see \eqref{purelyinterfacial}.
For this particular choice, we provide explicit formulas for the energy densities \eqref{W1}, \eqref{Gamma1}, \eqref{S10}, and \eqref{S11}, and we show that the relaxed energy $I$ is identically zero.

As it can be seen from the definitions of the energy densities $W_1$, $\Gamma_1$, $W_2$, and $\Gamma_2$, the functionals $I_1$ and $I_2$ are of the first order, meaning that only first-order derivatives enter in their definitions (the function spaces in \eqref{W1}, \eqref{Gamma1}, \eqref{S10}, and \eqref{S11} are of $SBV$ type).
Since $W=0$ and $\Psi_2=0$, the relaxed densities $W_2$ and $\Gamma_2$ in \eqref{S10} and \eqref{S11} are trivially equal to zero, so that the term $I_2$ in \eqref{S6} vanishes. 
Moreover, the relaxation procedure for obtaining $I_1$ can be carried out in the $SBV$ setting, as in the previous Sections \ref{sect:LHS} and \ref{sect:RHS}.
By invoking the results of \cite{nogap,S17}, one sees that the use of the strong convergence in $L^1$ in \eqref{S3} for an initial energy featuring $\Psi_1$ only is the same as using the weak convergence in $L^p$ considered in the previous Sections \ref{sect:LHS} and \ref{sect:RHS}, namely, there is no difference in considering either $\nabla u=A$ a.e. or $\int_{Q'} \nabla u=A$ in formula \eqref{W1}.

To compute the energy densities \eqref{W1} and \eqref{Gamma1} with the choice $\Psi_1(\lambda,\nu)=|\lambda\cdot\nu|$,we recall the definition of $\cl\Psi_1$ in \eqref{S9} and notice that it reads
\begin{equation}\label{1335}
\cl\Psi_1(\lambda,\eta)=\inf\{|\lambda\cdot(\eta|t)|:t\in\R{}\}=
\begin{cases}
|\lambda\cdot\tilde\eta| & \text{if $\lambda_3=0$}, \\
0 & \text{if $\lambda_3\neq0$}.
\end{cases}
\end{equation}

To show that $W_1=\Gamma_1=0$, we use the fact that $\overline{\Psi}_1(\lambda, \eta)$ vanishes whenever $\lambda_3 \neq 0$, so that jumps of infimizing approximations $u_n$ with non-zero third components have no energetic cost. We control the energetics cost of any necessary jumps with zero third-components by relegating them to transverse segments within the frames

\begin{equation}\label{frame2}
\cF_{n}:=Q'\setminus\Big(\frac{n-1}{n}\Big)Q',\qquad \cF_{n, \eta}:=Q'_\eta\setminus \Big(\frac{n-1}{n}\Big)Q'_\eta,
\end{equation}%
with $n$ a positive integer and $\eta \in \S1$. This approach was employed in \cite{nogap}, and we refer the reader to that article for any details omited here.

To show that $W_1(M) = 0$ for all $M \in \R{3\times2}$, we choose a constant $C > 0$ and, for each $n$ a function $v_n \in SBV(\cF_{n};\R3)$ such that
\begin{equation}\label{1705} 
v_n|_{\partial \cF_{n}} = 0, \quad \nabla v_n = M\quad \text{a.e. in}\; \cF_{n}, \quad \text{and}\quad | D^s v_n| \leq \frac{C}{n}.
\end{equation}

Next, we partition the shrunken square $(\tfrac{n-1}{n})Q'$ into $n$ thin rectangles $\mathcal{C}_{k,n}$, $k = 0, \ldots , n-1$, each of height $\frac{n-1}{n}$ and width $\frac{n-1}{n^2}$ (the width corresponding to the direction $e_1 = (1,0)$). 
Denoting the center of each rectangle by $c_{k,n}$, we define for each $n$ a function $u_n \in SBV(\omega;\R2)$ by 
\begin{equation}\label{1694} 
u_n(x) = \begin{cases}
v_n(x) & \text{if $x \in \mathcal{F}_{n}$,}\\
M(x - c_{k,n}) + \frac{(-1)^k}{n^2}e_3 & \text{if $x \in  c_{k,n}$, $k=1,\ldots,n -1$.}\\
\end{cases}
\end{equation}
It follows that
\begin{equation}\label{1700}
S(u_n) \subset S(v_n) \cup \partial (\tfrac{n-1}{n})Q' \cup \bigcup_{k=0}^{n-2} ( \partial \mathcal{C}_{k,n} \cap  \partial \mathcal{C}_{k+1,n})
\end{equation}
and we first consider $[u_n](x)$ when $x \in \partial (\tfrac{n-1}{n})Q'$. Using \eqref{1705} we have (to within a fixed choice of signs in front of each term)
\begin{equation}\label{1704}
[u_n](x)\cdot e_3 = \pm M(x - c_{k,n})\cdot e_3 \pm \frac{1}{n^2}= \pm (x - c_{k,n})\cdot M^\top e_3 \pm \frac{1}{n^2},
\end{equation}
so that $[u_n](x)\cdot e_3 = 0$ if and only if $M^\top e_3 \neq 0$ and $x$ is on the line $\ell = \{ y \in \R2: ( y - c_{k,n})\cdot M^\top e_3 \pm \frac{1}{n^2} = 0\}$ in $\R2$ whose distance from $c_{k,n}$ is $(n^2|M^\top e_3|)^{-1} = O(n^{-2}).$
Because the distance from $c_{k,n}$ to $\partial (\tfrac{n-1}{n})Q'$ is at least $\frac{n-1}{2n^2} = O(n^{-1})$, it follows that for $n$ sufficiently large the line $\ell$ intersects $\partial (\tfrac{n-1}{n})Q'$ at exactly two points. 
We conclude from \eqref{1335} that, whether or not $M^\top e_3 \neq 0$, for $n$ sufficiently large
\begin{equation}\label{1711}
\overline{\Psi}_1( [u_n](x), \nu(u_n)(x)) = 0\quad \text{for $\cH^1$-a.e.\@ $x\in \partial (\tfrac{n-1}{n})Q'$.}
\end{equation}

We consider next a point $x \in  \bigcup_{k=0}^{n-2} ( \partial \mathcal{C}_{k,n} \cap  \partial \mathcal{C}_{k+1,n})$ and use \eqref{1694} to compute
\begin{equation}\label{1715}
[u_n](x)\cdot e_3 = \pm M(c_{k,n} - c_{k+1,n})\cdot e_3 \pm \frac{1}{n^2} = \pm(c_{k,n} - c_{k+1,n})\cdot M^\top e_3 \pm \frac{1}{n^2},
\end{equation}
which is zero only if $M^\top e_3 \neq 0$. 
However, $|c_{k,n} - c_{k+1,n}| = \frac{n-1}{n^2} = O(n^{-1})$ so that for $n$ sufficiently large $[u_n](x)\cdot e_3 \neq 0$ for every $x \in \bigcup_{k=0}^{n-2} ( \partial \mathcal{C}_{k,n} \cap  \partial \mathcal{C}_{k+1,n})$, and we conclude 
\begin{equation}\label{1722}
\overline{\Psi}_1( [u_n](x), \nu(u_n)(x)) = 0\quad \text{for $\cH^1$-a.e.\@ $x\in \bigcup_{k=0}^{n-2} (\partial \mathcal{C}_{k,n} \cap \partial \mathcal{C}_{k+1,n})$}
\end{equation}
and that, by \eqref{1705}, \eqref{1694} and \eqref{1700}, 
\begin{equation}\label{1726}
\int_{Q'\cap S(u_n)}\Psi_1([u_n], \nu(u_n))\, \de \cH^1(x_\alpha)  = \int_{Q'\cap S(v_n)}\Psi_1([v_n], \nu(v_n))\, \de \cH^1(x_\alpha) = |D^s v_n| = O\left(\frac{1}{n}\right).
\end{equation}
Because $u_n$ is admissible in \eqref{W1} we conclude that $W_1(M) = 0$.

To show that $\Gamma_1(\lambda, \eta) = 0$ for all $\lambda \in \R3$ and $\eta \in \S1$, we note first that the mapping $\gamma_{\lambda, \eta}: Q'_\eta \to \R3$ is admissible in \eqref{Gamma1}, so that
\begin{equation}\label{1734}
0  \leq \Gamma_1(\lambda, \eta)  \leq \int_{Q'_\eta\cap S(\gamma_{\lambda, \eta})}\Psi_1([\gamma_{\lambda, \eta}], \nu(\gamma_{\lambda, \eta}))\, \de \cH^1(x_\alpha) = \Psi_1(\pm \lambda, \eta).
\end{equation}
In particular, if $\lambda_3 \neq 0$, then \eqref{1734} and \eqref{1335} yield $\Gamma_1(\lambda, \eta) = 0.$

Suppose now that $\lambda_3 = 0$.  
With $\cF_{n, \eta}$ defined as in \eqref{frame2}, we define $u_n: Q'_\eta \to \R3$ by
\begin{equation}\label{1742} u_n(x) = \begin{cases}
\gamma_{\lambda, \eta} & \text{if $x \in \mathcal{F}_{n}$,}\\
\gamma_{\lambda, \eta} - \frac{1}{n}e_3 & \text{if $x \in  (\tfrac{n-1}{n})Q'_\eta$ and $x\cdot \eta \leq 0$,} \\
\gamma_{\lambda, \eta} + \frac{1}{n}e_3 & \text{if $x \in  (\tfrac{n-1}{n})Q'_\eta$ and $x\cdot \eta \geq 0$.} \\
\end{cases}
\end{equation}
It follows that $S(u_n) \subset \partial (\tfrac{n-1}{n})Q'_\eta \cup \{ x \in Q'_\eta : x\cdot \eta= 0\}$. 
If $ x \in Q'_\eta \cap S(u_n),$ then $[u_n](x) = [\gamma_{\lambda,\eta}] + \frac{m(x)}{n}e_3$ with $m(x) \in \{ 0, 1, -1, 2, -2\}$ and
\begin{equation}\label{1743}
m(x) = 0\quad \text{if and only if}\quad x\cdot\eta = 0\quad \text{and}\quad |x| \in \bigg[ \frac{n-1}{2n}, \frac{1}{2}\bigg].
\end{equation}
Because $[\gamma_{\lambda, \eta}](x) \in \{\lambda, -\lambda, 0\}$ and $\lambda\cdot e_3 = \pm\lambda_3 = 0,$ it follows that $[u_n](x) \cdot e_3 = 0$ if and only if $m(x) = 0$, i.e., 
\begin{equation}\label{1744}
[u_n](x)\cdot e_3 = 0 \quad \text{if and only if}\quad x\cdot \eta = 0\quad \text{and}\quad |x| \in \bigg[\frac{n-1}{2n}, \frac{1}{2}\bigg].
\end{equation}
We conclude from \eqref{1335} that: $\cl\Psi_1([u_n](x), \nu(u_n)(x))\neq 0$ if and only if $x\cdot \eta = 0$ and $|x| \in \left[ \frac{n-1}{2n}, \frac{1}{2}\right]$, so that
\begin{equation}\label{1745}
\begin{split}
0  \leq \Gamma_1(\lambda, \eta) \leq & \int_{Q'_\eta\cap S(u_n)} \Psi_1([u_n], \nu(u_n))\, \de \cH^1(x_\alpha) \\
= & \int_{Q'_\eta\cap\left\{x\cdot \eta = 0 \text{ and } |x| \in \left[ \frac{n-1}{2n}, \frac{1}{2}\right]\right\}} \Psi_1([\gamma_{\lambda, \eta}], \nu(\gamma_{\lambda,\eta})\, \de \cH^1(x_\alpha) \leq \frac{|\lambda |}{n}.
\end{split}
\end{equation}
Because each $u_n$ is admissible in \eqref{Gamma1}, $\{u_n\}$ is an infimizing sequence and $\Gamma_1(\lambda, \eta) = 0.$

\section{Conclusions}\label{conclusions}

In this paper we have studied a problem that involves both dimension reduction and introduction of disarrangements.
From the point of view of energetics, this entails two relaxation processes, so that the order in which they are performed is relevant for the structure of the final, doubly relaxed energy functional.
In this respect, we applied the two relaxation processes one after the other in both orders and we obtained two doubly relaxed energy functionals, those in \eqref{201} and in \eqref{301}.

At the level of generality considered in Theorems \ref{secondleft} and \ref{secondright}, we did not undertake a comparison of these two formulas.
Nonetheless, we compared them in a special case which is relevant to the multiscale nature of the geometry of structured deformations, namely we considered an initial energy which takes into account only the normal component of the jumps.
In this case, we were able to prove that the doubly relaxed energy functionals are the same, see Proposition \ref{S1}.
Moreover, we compared our procedure with one that has been studied by Matias and Santos in \cite{MS}: here, the dimension reduction and the relaxation to structured deformations are performed simultaneously.
With the same choice of a purely interfacial initial energy, we computed the relaxed energy in the context of \cite{MS} and we proved that it is identically equal to zero.
This suggests looking at different scalings in the vanishing thickness parameter $\eps$, in particular, looking for higher-order terms in the expansion by $\Gamma$-convergence in the sense of \cite{AB}. 

It is worth noticing that, in spite of the technical differences in the three relaxation procedures carried out, the final relaxed energies are all defined on the same type of mathematical objects, namely a structured deformation and a director, defined on the cross--section $\omega$.
To see this, one can compare the triple $(\cl g,\cl G,\cl d)\in SBV(\omega;\R3){\times}L^1(\omega;\R{3{\times}2}){\times}L^p(\omega; \R3)$ in Theorems \ref{secondleft} and \ref{secondright} with the triple $(g,b,G)\in BV^2(\omega;\R3){\times}BV(\omega;\R3){\times}BV(\omega;\R{3\times2})$ in Theorem \ref{S4}.

It is natural to conjecture that the relaxation described in Theorem \ref{S4} yields a lower energy than those provided by Theorems \ref{secondleft} and \ref{secondright}.
In this regard, the results contained in \cite{S17} provide a useful tool for studying this conjecture.
In view of the results of Sections \ref{sect:comp_ex} and \ref{sect:6}, we can answer affirmatively to the conjecture in the case of a particular choice of the initial energy. 

Finally, we remark that a common feature of all three approaches is the introduction of constraints on the admissible disarrangements, namely that the normal to the jump set be aligned with the two-dimensional approximating object.
This is enforced by the condition $\nu(u_n)\cdot e_3=0$ in \eqref{107a}, by the condition $\int_Q \nabla u\,\de x=(B^{\backslash 3}|Ae_3)$ in \eqref{2015}, and by the conditions cited in \cite[Remark 1.5]{MS}.

\medskip

\noindent\textbf{Acknowledgements.} The authors warmly thank the Departamento de Matem\'atica at Instituto Superior T\'ecnico in Lisbon, the Departamento de Matem\'atica at Universidade de \'Evora, the Center for Nonlinear Analysis at Carnegie Mellon University in Pittsburgh, the Fakult\"at f\"ur Mathematik at Technische Universit\"at M\"unchen, SISSA in Trieste, and the Dipartimento di Ingegneria Industriale of the Universit\`a di Salerno, where this research was developed.

The research of J.M.\@ was partially supported by the Funda\c{c}\~{a}o para a Ci\^{e}ncia e a Tecnologia through grant UID/MAT/04459/2013, by the Center for Nonlinear Analysis at Carnegie Mellon University in Pittsburgh, and by the Gruppo Nazionale per l'Analisi Matematica, la Probabilit\`a e le loro Applicazioni (GNAMPA) of the Istituto Nazionale di Alta Matematica (INdAM).
The research of M.M.\@ was partially funded by the ERC Advanced grant \emph{Quasistatic and Dynamic Evolution Problems in Plasticity and Fracture} (Grant agreement no.: 290888) and by the ERC Starting grant \emph{High-Dimensional Sparse Optimal Control} (Grant agreement no.: 306274).
M.M.\@ is a member of the Gruppo Nazionale per l'Analisi Matematica, la Probabilit\`a e le loro Applicazioni (GNAMPA) of the Istituto Nazionale di Alta Matematica (INdAM).


\begin{thebibliography}{99}
\bibitem{numberone} {\sc M. Alam and S. Luding}: \emph{First normal stress difference and crystallization in a dense sheared granular fluid}. Phys.  Fluids, \textbf{15} (2003), 2298-2312.

\bibitem{AL} {\sc G.  Alberti}: \emph{A Lusin-type Theorem for gradients}. J. Funct. Anal.,  \textbf{100} (1991), 110-118.

\bibitem{AFP} {\sc L.  Ambrosio, N.  Fusco, and D.  Pallara}: \emph {Functions of Bounded Variation and Free Discontinuity Problems}. Oxford University Press, 2000.

\bibitem{AMT} {\sc L. Ambrosio, S. Mortola, and V. M. Tortorelli}: \emph{Functionals with linear growth defined on vector valued $BV$ functions}. J. Math. Pures et Appl., \textbf{70} (1991), 269-323.

\bibitem{angelillo} {\sc M. Angelillo}: \emph{Constitutive relations for no-tension materials}. Meccanica \textbf{28} (1993), 195-202.

\bibitem{AB} {\sc G. Anzellotti and S. Baldo}: \emph{Asymptotic development by $\Gamma$-convergence}. Appl. Math. Optim., \textbf{27}(2) (1993), 105-123.

\bibitem {BBBF} {\sc A. C. Barroso, G. Bouchitt\'{e}, G. Buttazzo, and I. Fonseca}, \emph{Relaxation of bulk and interfacial energies}. Arch. Rational Mech. Anal., \textbf{135}(2) (1996), 107-173.

\bibitem{nogap} {\sc A. C. Barroso, J. Matias, M. Morandotti, and D. R. Owen}: \emph{Explicit Formulas for Relaxed Energy Densities Arising from Structured Deformations}. Math. Mech. Complex Syst., \textbf{5}(2) (2017), 163-189.

\bibitem{BFM} {\sc G. Bouchitt\'e, I. Fonseca, and L. Mascarenhas}: \emph{Bending moment in membrane theory}. J. Elasticity, \textbf{73}(1-3) (2003), 75-99. 

\bibitem{BFM1} {\sc  G. Bouchitt\'e, I. Fonseca, and L. Mascarenhas}: \emph{The Cosserat vector in membrane theory: a variational approach}. J. Convex Anal., \textbf{16}(2) (2009), 351-365.

\bibitem{Braides} {\sc A. Braides}: $\Gamma$-convergence for beginners. Oxford Lecture Series in Mathematics and its Applications, 22. Oxford University Press, Oxford, 2002. 

\bibitem{BF2001} {\sc A. Braides and I. Fonseca}: \emph{Brittle Thin Films}. Appl. Math. Optim., \textbf{44} (2001), 299-323.

\bibitem{CLT1} {\sc M. Carriero, A. Leaci, and F. Tomarelli}: \emph{A second order model in image segmentation: Blake and Zisserman functional}. Prog. Nonlinear Differ. Equ. Appl., \textbf{25} (1996), 57-72.

\bibitem{CLT2} {\sc M. Carriero, A. Leaci, and F. Tomarelli}: \emph{Second order variational problems with free discontinuity and free gradient discontinuity}. In: Calculus of Variations: Topics from the Mathematical Heritage of E. De Giorgi. Quad. Mat., vol. \textbf{14} (2004), 135-186.

\bibitem{CF} {\sc R. Choksi and I. Fonseca}: \emph{Bulk and Interfacial Energies for Structured Deformations of Continua}. Arch. Rational Mech. Anal., \textbf{138} (1997), 37-103.

\bibitem{DGA}{\sc E. De Giorgi and L. Ambrosio}: \emph{Un nuovo tipo di funzionale del calcolo delle variazioni}. Atti Accad. Naz. Lincei,  \textbf{82} (1988), 199-210.

\bibitem{DGF} {\sc E. De Giorgi and T. Franzoni}: \emph{Su un tipo di convergenza variazionale}. Atti Accad. Naz. Lincei Rend. Cl. Sci. Fis. Mat. Natur. (8) \textbf{58}(6) (1975), 842-850.

\bibitem{DM} {\sc G. Dal Maso}: An introduction to $\Gamma$-convergence. Progress in Nonlinear Differential Equations and their Applications, 8. Birkh\"auser Boston, Inc., Boston, MA, 1993.

\bibitem{DPO95} {\sc G. Del Piero and D. R. Owen}: \emph{Integral-gradient formulae for structured deformations}. Arch. Rational Mech. Anal. \textbf{131}(2) (1995), 121-138.

\bibitem{DPZ} {\sc L. Deseri, M. Piccioni, and D. Zurlo}: \emph{Derivation of a new free energy for biological membranes}. Contin. Mech. Thermodyn., \textbf{20}(5) (2008), 255-273. 

\bibitem{DO} {\sc L. Deseri and D. R. Owen}: \emph{Stable disarrangement phases of elastic aggregates: a setting for the emergence of no-tension materials with non-linear response in compression}. Meccanica \textbf{49}(12) (2014), 2907-2932.

\bibitem{DO1} {\sc L. Deseri and D. R. Owen}: \emph{Toward a field theory for elastic bodies undergoing disarrangements}. Essays and papers dedicated to the memory of Clifford Ambrose Truesdell III. Vol. I. J. Elasticity \textbf{70}(1-3) (2003), 197-236.

\bibitem{FM}{\sc I. Fonseca and S. M\"uller }: \emph{Quasi-convex integrands and lower semicontinuity in $L^1$}. SIAM J. Math. Anal. \textbf{23} (1992), 1081-1098.

\bibitem{K} {\sc D. Khakhar}: \emph{Rheology and mixing of granular materials}. Macromol. Mater. Eng., \textbf{296} (2011), 278-289.

\bibitem{KS4}{\sc R. V. Kohn and G. Strang}: \emph{Optimal design in elasticity and plasticity}. Int. Journal for Numerical Methods in Engineering, \textbf{22}, (1986), 183-188.

\bibitem{LDR95} {\sc H. Le Dret and A. Raoult}: \emph{The nonlinear membrane model as variational limit of nonlinear three-dimensional elasticity}. J. Math. Pures Appl., \textbf{74} (1995), 549-578.

\bibitem{LDR96} {\sc H. Le Dret and A. Raoult}: \emph{The membrane shell model in nonlinear elasticity: A variational asymptotic derivation}. J. Nonlinear Sci., \textbf{6}(1) (1996), 59-84.

\bibitem{LSZ} {\sc M. Lucchesi, M. \v{S}ilhav\'{y}, and N. Zani}: \emph{A new class of equilibrated stress fields for no-tension bodies}. J. Mech. Mater. Struct., \textbf{1}(3) (2006) 503-539.

\bibitem{M}{\sc P.  Marcellini}: \emph {Approximation of quasiconvex functions, and lower semicontinuity of multiple integrals}. Manuscripta Math., \textbf{51}(1-3) (1985), 1-28.

\bibitem{M07} {\sc J. Matias}: \emph{Differential inclusions in $SBV_0(\Omega)$ and applications to the Calculus of Variations}, J. Convex Analysis \textbf{14}(3) (2007), 465-477.

\bibitem{MMZ} {\sc J. Matias, M. Morandotti, and E. Zappale}: \emph{Optimal Design of Fractured Media with Prescribed Macroscopic Strain}. J. Math. Anal. Appl., \textbf{449} (2017), 1094-1132.

\bibitem{MS} { \sc J.  Matias and P. M.  Santos}: \emph{A dimension reduction result in the framework of structured deformations}. Appl. Math. Optim., \textbf{69}(3) (2014), 459-485.

\bibitem{Mue} {\sc N. Mueggenburg}: \emph{Behavior of granular materials under cyclic shear}. Phys. Rev. E, \textbf{71} (2005), 031301.

\bibitem{owen} {\sc D. R. Owen}: \emph{Elasticity with gradient-disarrangements: a multiscale perspective for strain-gradient theories of elasticity and of plasticity}. J. Elasticity \textbf{127}(1) (2017), 115-150.

\bibitem{OP15} {\sc  D. R. Owen and R. Paroni}: \emph{Optimal flux densities for linear mappings and the multiscale geometry of structured deformations}. Arch. Rational Mech. Anal. \textbf{218}(3) (2015), 1633-1652.

\bibitem{R2}{\sc D. Raabe, M. Sachtleber, Z. Zhao, F. Roters, and S. Zaefferer}: \emph{Micromechanical and macromechanical effects in grain scale polycrystal plasticity. Experimentation and simulation}. Acta Materialia, \textbf{49}  (2001) 3433-3441.

\bibitem{S17} {\sc M. \v{S}ilhav\'y}: \emph{The general form of the relaxation of a purely interfacial energy for structured deformations}. Math. Mech. Complex Syst., \textbf{5}(2) (2017), 191-215.

\bibitem{V} {\sc A. I. Vol'pert}: \emph{The space $BV$ and quasilinear equations}. Mat. Sbornik \textbf{73} (1967), 255-302.

\bibitem{VH} {\sc A. I. Vol'pert and S. I. Hudjaev}: \emph{Analysis in Classes of Discontinuous Functions and the Equations of Mathematical Physics}. Nijhoff, 1985.
\end{thebibliography}
\end{document}